\documentclass[11pt]{amsart}

\usepackage[english]{babel}

\usepackage[a4paper,top=2cm,bottom=2cm,left=3cm,right=3cm,marginparwidth=1.75cm]{geometry}

\usepackage{amsmath}
\usepackage{graphicx}
\usepackage[colorlinks=true, allcolors=blue]{hyperref}
\usepackage{amsthm}
\usepackage{amssymb}
\usepackage{mathtools}
\usepackage{enumitem}

\newtheorem{thm}{Theorem}[section]
\newtheorem{cor}[thm]{Corollary}
\newtheorem{lem}[thm]{Lemma}
\newtheorem{prop}[thm]{Proposition}
\theoremstyle{definition}

\theoremstyle{remark}
\newtheorem{rem}[thm]{Remark}
\newtheorem{question}[thm]{Question}

\begin{document}

\title[Z2-Thurston norm in Sol manifolds]{$\mathbb{Z}_2$-Thurston norm in Sol manifolds and embeddability of non-orientable surfaces}

\author[Xiaoming Du \and Weibiao Wang]{Xiaoming Du$^1$ \and Weibiao Wang$^{2}$}
\address{$^1$ School of Mathematics, South China University of Technology, Guangzhou, 510640, PR China}
\email{scxmdu@scut.edu.cn}
\address{$^2$ School of Mathematics and Statistics, Central South University, Changsha, 410083, PR China
}
\email{wangweibiao@csu.edu.cn}



\begin{abstract}
For every Sol manifold $M$, we determine the $\mathbb{Z}_2$-Thurston norm of every element in $H_2(M;\mathbb{Z}_2)$.
Each Sol manifold is either a torus bundle over the circle or a torus semi-bundle, thus corresponds to a torus map.
We discuss the action of this torus map on a curve complex for the torus, whose edges connect curve classes of intersection number 2. 
For torus bundles over the circle, the $\mathbb{Z}_2$-Thurston norm of any $\mathbb{Z}_2$-homology class equals either zero or the minimum translation distance under the action; and for torus semi-bundles, it equals either zero or the translation distance of a specific curve class.
Moreover, we construct incompressible surfaces to realize all the $\mathbb{Z}_2$-homology classes.
	
As a consequence, for any torus bundle over the circle or torus semi-bundle, we determine which non-orientable closed surfaces can be embedded in it.
\end{abstract}

\maketitle

\section{Introduction}

For a connected closed 3-manifold $M$, let us consider the following question. 

\begin{question}
	\label{question}
	Given a $\mathbb{Z}_2$-homology class $\alpha\in H_2(M;\mathbb{Z}_2)$, how can we find an embedded surface in $M$ that is the ``simplest" one to represent $\alpha$?
\end{question}

It is known that for each $\alpha\in H_2(M;\mathbb{Z}_2)$, there is an embedded closed surface $F$ with $[F]=\alpha$ (see \cite[Section 2]{BW1969}).
Such a surface $F$ may be disconnected.
One can make it connected by performing connected sums, but this will also make $F$ compressible and thus not the ``simplest" in some sense.
How to measure the complexity of $F$?
A possible way is to consider
\begin{displaymath}
	\chi_-(F)=\sum_{i=1}^{n}\max\{0,-\chi(F_i)\},
\end{displaymath}
where $F_1,\cdots,F_n$ are all the components of $F$ and $\chi(\cdot)$ denotes the Euler characteristic.
Jaco, Rubinstein and Tillmann \cite{JRT2013} introduced the \textbf{$\mathbb{Z}_2$-Thurston norm} of $\alpha\in H_2(M;\mathbb{Z}_2)$: 
\begin{displaymath}
	||\alpha||_{\mathbb{Z}_2} =\min \{\chi_-(F): [F]=\alpha\},
\end{displaymath}
which is an analogue of the Thurston norm for elements in $H_2(M;\mathbb{Z})$ \cite{Thurston1986}.
So to answer Question \ref{question} is to find an embedded surface $F$ realizing the $\mathbb{Z}_2$-Thurston norm. 
Moreover, in many cases such a surface can be \textbf{$\mathbb{Z}_2$-taut}, i.e.,  $||[F]||_{\mathbb{Z}_2}=-\chi(F)$ and each component of $F$ is neither a sphere nor a projective plane.

According to Thurston's geometrization, there are eight kinds of geometric 3-manifolds, among which six can be Seifert fibered, one is Sol, and one is hyperbolic.
For orientable, closed, irreducible, Seifert fibered 3-manifolds with orientable base surfaces, a method to compute $\mathbb{Z}_2$-Thurston norms was given in \cite{Du2022}.
This paper is devoted to determining $\mathbb{Z}_2$-Thurston norms for all orientable closed Sol manifolds and finding $\mathbb{Z}_2$-taut surfaces realizing them.

Each orientable closed Sol manifold is either a torus bundle over $S^1$ or a torus semi-bundle \cite[Theorem 5.3]{Scott1983}.
We consider these two kinds of manifolds, including torus bundles over $S^1$ which are non-orientable or admit $\mathbb{E}^3$ or Nil geometry.
If $M$ is such a manifold, we will see that $H_2(M;\mathbb{Z}_2)$ can be expressed as a direct sum $G_0\oplus G_1$, where $G_0$ is generated by a torus or two Klein bottles and thus has a trivial $\mathbb{Z}_2$-Thurston norm; $G_1$ is determined by a torus map.
For any $\alpha\in H_2(M;\mathbb{Z}_2)$ with projection $\alpha_1\in G_1$, the $\mathbb{Z}_2$-Thurston norm $||\alpha||_{\mathbb{Z}_2}$ equals $||\alpha_1||_{\mathbb{Z}_2}$.
In fact, the surfaces realizing $||\alpha||_{\mathbb{Z}_2}$ and $||\alpha_1||_{\mathbb{Z}_2}$ can be converted to each other by a surgery (see Proposition \ref{prop:sum}).
Therefore, the only task is to determine the $\mathbb{Z}_2$-Thurston norms in $G_1$.

If $M$ is a torus bundle over $S^1$, i.e., the mapping torus of a self-homeomorphism $f$ of the torus $T^2$, then $G_0=\mathbb{Z}_2$, generated by a torus fiber.
For a $\mathbb{Z}_2$-taut and incompressible surface $F$ with $[F]\in G_1\backslash\{0\}$, we cut $M$ along a torus fiber. 
Then $F$ becomes an incompressible surface in a thickened torus $T^2\times I$, whose boundary can be assumed to have the form of $(c\times\{0\})\cup(f(c)\times\{1\})$ with $c$ an essential simple closed curve on $T^2$.
According to \cite{Przytycki2019}, such a surface corresponds to the shortest path connecting $c,f(c)$ in the intersection number 2 curve complex $\mathcal{C}^{i=2}(T^2)$, which is a forest of three trees (see Section \ref{sect:complex}).
By considering the action of $f$ on $\mathcal{C}^{i=2}(T^2)$ and computing the minimum distance between $c$ and $f(c)$, we can finally determine the value of  $||[F]||_{\mathbb{Z}_2}$ from the matrix of $f$.
See Theorem \ref{thm:bundle} and Corollary \ref{cor:l} for the conclusion.

If $M$ is a torus semi-bundle, then $M$ can be constructed by gluing two copies of the twisted $I$-bundle over the Klein bottle via a torus map $f'$.
In this case $G_0$ is $\mathbb{Z}_2\oplus\mathbb{Z}_2$, generated by the two Klein bottles.
For a $\mathbb{Z}_2$-taut and incompressible surface $F$ with $[F]\in G_1\backslash\{0\}$, we cut $M$ along two torus fibers. 
Then we get three pieces: a thickened torus $T^2\times I$, and two copies $N_1,N_2$ of the twisted $I$-bundle over the Klein bottle. 
We may assume that $F\cap N_i$ ($i=1,2$) is essential in $N_i$ and $F\cap (T^2\times I)$ is incompressible in $T^2\times I$.
Incompressible surfaces and essential surfaces in irreducible Seifert fibered 3-manifolds are well-studied. 
They are isotopic to pseudo-vertical or pseudo-horizontal surfaces \cite{Frohman1986,Rannard1996,KN2023},
which can be encoded by several parameters \cite{Du2022}.
By analyzing pseudo-vertical surfaces and pseudo-horizontal surfaces in the three pieces, we show that $||[F]||_{\mathbb{Z}_2}$ can be determined by the matrix of $f'$.
See Theorem \ref{thm:semi-bundle} for the conclusion.

As a consequence, we answer the following question for all torus bundles over $S^1$ (including non-orientable ones) and all torus semi-bundles. 
See Theorems \ref{thm:mog}, \ref{thm:meg-mog}, \ref{thm:meg}.
\begin{question}
	For a given 3-manifold, which non-orientable closed surfaces can be embedded in it?
\end{question}

It is known that for any closed 3-manifold $M$, there exists an embedded non-orientable closed surface if and only if $H_2(M;\mathbb{Z}_2)$ is nontrivial \cite[Proposition 2.2]{BW1969}.
But generally it is hard to answer the above question.
Bredon and Wood \cite{BW1969} solved this problem for the following 3-manifolds: lens spaces, $F\times S^1$ with $F$ being any orientable closed surface, and finite connected sums of them.
End \cite{End1992} and Rannard \cite{Rannard1996} gave an answer for $F' \times S^1$ where $F'$ is any non-orientable closed surface.

This paper is organized as follows.
In Section \ref{sect:complex} we introduce the intersection number 2 curve complex $\mathcal{C}^{i=2}(T^2)$.
In Section \ref{sect:incompressible-surface}, we discuss the structure of incompressible surfaces and essential surfaces in the thickened torus, the solid torus, and the twisted $I$-bundle over the Klein bottle.
Then in Section \ref{sect:norm} we determine $\mathbb{Z}_2$-Thurston norms in all torus bundles over $S^1$ and torus semi-bundles (in particular, all Sol manifolds), and construct surfaces realizing them.
Finally in Section \ref{sect:embeddability}, we answer the question of which non-orientable closed surfaces can be embedded in these manifolds.

We work in the category of smooth manifolds.
Some constructions will be presented piecewise linearly, but it is not hard to see that they can all be made smooth.
All intersections of submanifolds are assumed to be transverse.

\textbf{Acknowledgements} Weibiao Wang is supported by the National Natural Science Foundation of China (grant No. 12501089), and the Natural Science Foundation of Hunan Province, China (grant No. 2026JJ60114).

\section{The intersection number 2 curve complex}\label{sect:complex}

Fix two oriented simple closed curves $\lambda$ and $\mu$ on the torus $T^2$ such that they generate the homology group:
\begin{displaymath}
	H_1(T^2;\mathbb{Z})=\mathbb{Z}\langle[\lambda]\rangle\oplus\mathbb{Z}\langle[\mu]\rangle.
\end{displaymath}
An essential simple closed curve $c$ on $T^2$ is said to be a \textbf{$p/q$-curve} if with an arbitrary orientation, the element $[c]\in H_1(T^2; \mathbb{Z})$ represented by $c$ equals $p[\lambda]+q[\mu]$.
Here $p,q$ are coprime integers and $p/q$ is called the \textbf{slope} of $c$.
It is well-known that two essential simple closed curves on $T^2$ are isotopic if and only if they have the same slope (see \cite[\S 2.C]{Rolfsen2003}), and thus there is a 1-1 correspondence:
\begin{displaymath}
	\begin{matrix}
		\mathbb{Q}\cup\{\infty\} &\xleftrightarrow{1-1} &\{\text{isotopy classes of essential simple closed curves on }T^2\}.\\
		p/q & \xleftrightarrow{\quad} & \text{the isotopy class of }p/q\text{-curves}
	\end{matrix}
\end{displaymath}
In this paper they are identified and the isotopy class of $p/q$-curves will be written briefly as $p/q$, where $p,q$ are always assumed to be coprime to each other.

For two isotopy classes of simple closed curves on a surface, their \textbf{geometric intersection number} is the minimum number of intersection points between two representative curves of them.
On $T^2$, the geometric intersection number of the $p/q$-curve class and the $p'/q'$-curve class, denoted by $i(p/q,p'/q')$, equals $|pq'-p'q|$ \cite[\S 1.2.3]{FM2011}.

The \textbf{intersection number 2 curve complex} of $T^2$, denoted by $\mathcal{C}^{i=2}(T^2)$, is the graph defined by the following data.
\begin{itemize}
	\item The vertex set is $\mathbb{Q}\cup\{\infty\}$. 
	That is to say, there is one vertex for each isotopy class of essential simple closed curves on $T^2$.
	\item The edge set is
	\begin{displaymath}
		\{\{p/q,p'/q'\}: p/q,p'/q'\in\mathbb{Q}\cup\{\infty\}, |pq'-p'q|=2\}.
	\end{displaymath}
	That is to say, two vertices are connected by an undirected edge if and only if the corresponding curve classes have a geometric intersection number of $2$.
\end{itemize}
It was shown in \cite[Proposition 2.2]{Przytycki2019} that $\mathcal{C}^{i=2}(T^2)$ has three connected components and they are all trees. 
We denote them by $T_{1/0},T_{0/1},T_{1/1}$, where the vertex set of $T_{j/k}$ ($j/k\in\{1/0,0/1,1/1\}$) is 
\begin{displaymath}
	V_{j/k}=\{p/q\in\mathbb{Q}\cup\{\infty\}: p \equiv j, \, q \equiv k \bmod 2\},
\end{displaymath}
and the edge set of $T_{j/k}$ consists of all edges in $\mathcal{C}^{i=2}(T^2)$ with endpoints in $V_{j/k}$.

If we replace the edge set in the definition of $\mathcal{C}^{i=2}(T^2)$ by 
\begin{displaymath}
	\{\{p/q,p'/q'\}: p/q,p'/q'\in\mathbb{Q}\cup\{\infty\}, |pq'-p'q|=1\},
\end{displaymath}
then we get the intersection number 1 curve complex of $T^2$, denoted by $\mathcal{C}^{i=1}(T^2)$, which may be more familiar to readers.
According to \cite[\S 4.1.1]{FM2011}, $\mathcal{C}^{i=1}(T^2)$ is a Farey graph, which can be combinatorially embedded in $\mathbb{H}^2\cup\partial\mathbb{H}^2$, the hyperbolic plane with the boundary at infinity, and gives the Farey tessellation (see Figure \ref{fig:Farey}, in the disk model of $\mathbb{H}^2$). 
Here the vertex set of $\mathcal{C}^{i=1}(T^2)$ is identified with $\mathbb{Q}\cup\{\infty\}\subset \mathbb{R}\cup\{\infty\}\cong\partial\mathbb{H}^2$, and the edges are realized as geodesics.
Similarly, each component of $\mathcal{C}^{i=2}(T^2)$ can be combinatorially embedded in $\mathbb{H}^2\cup\partial\mathbb{H}^2$. 
Figure \ref{fig:trees} shows $T_{1/0},T_{0/1},T_{1/1}$, in green, blue, and red respectively.

\begin{figure}[htbp]
	\centering
	\includegraphics[scale=0.3]{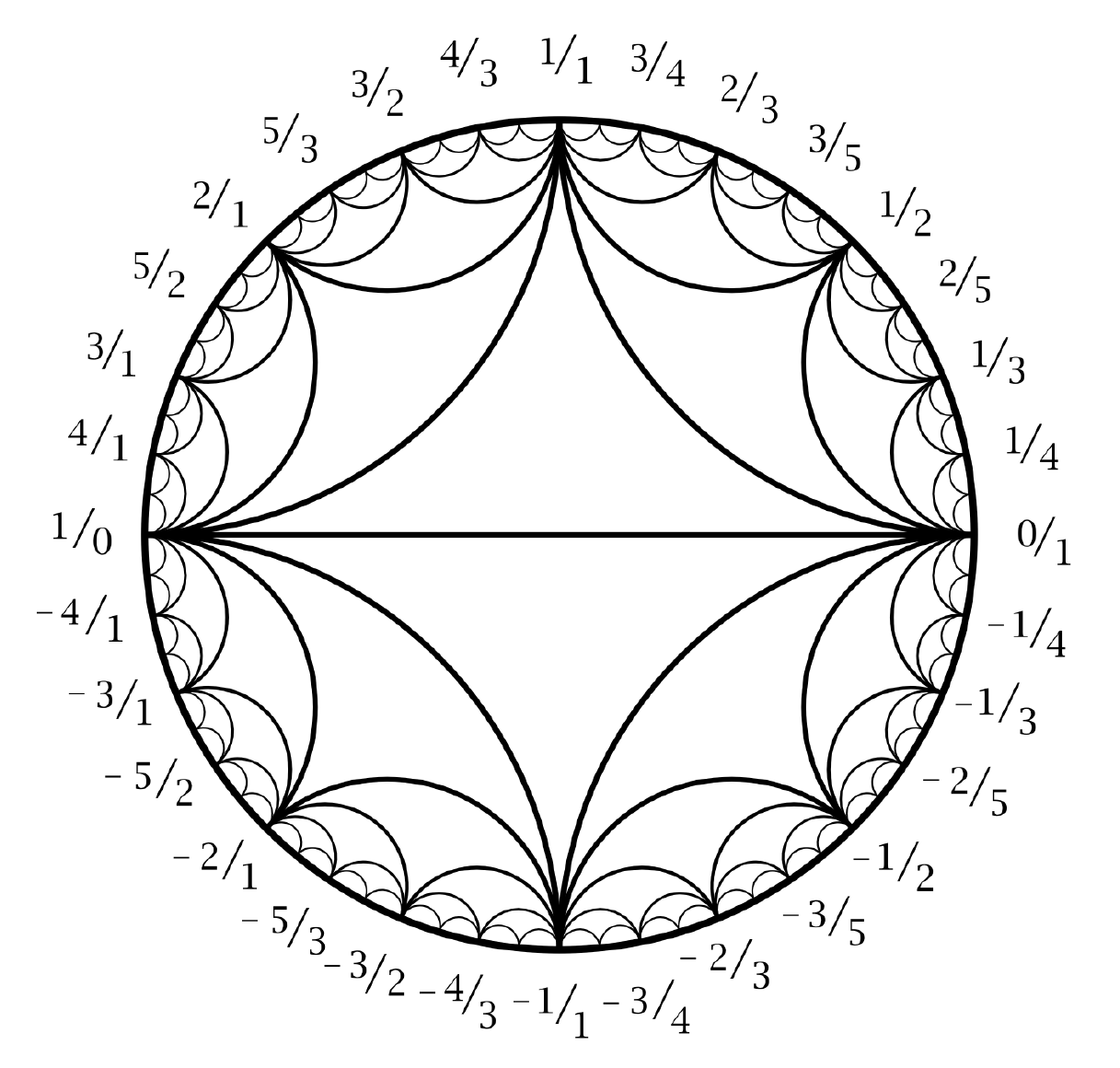}
	\caption{$\mathcal{C}^{i=1}(T^2)$: a Farey graph.}
	\label{fig:Farey}
\end{figure}

\begin{figure}[htbp]
	\centering
	\includegraphics[scale=1.5]{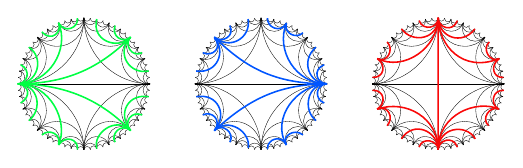}
	\caption{The three components of $\mathcal{C}^{i=2}(T^2)$.}
	\label{fig:trees}
\end{figure}

Denote the group of automorphisms (self-diffeomorphisms) of $T^2$ by ${\rm Aut}(T^2)$.
Each $f\in{\rm Aut}(T^2)$ induces an automorphism $f_*$ of $H_1(T^2;\mathbb{Z})$ with a corresponding matrix in $GL(2,\mathbb{Z})$, denoted by $A_f$: 
\begin{displaymath}
	(\begin{matrix}
		f_*[\lambda] & f_*[\mu]
	\end{matrix})
	=(\begin{matrix}
		[\lambda] &[\mu]
	\end{matrix})\,A_f.
\end{displaymath}
This gives an action of ${\rm Aut}(T^2)$ on $\mathbb{Q}\cup\{\infty\}$ as well as on $\mathcal{C}^{i=2}(T^2)$, which factors through $GL(2,\mathbb{Z})$.
Note that the mapping class group of $T^2$, i.e., ${\rm Aut}(T^2)/$isotopy, is isomorphic to $GL(2,\mathbb{Z})$, via $[f]\mapsto A_f$ (see \cite[\S 2. D]{Rolfsen2003}).

\begin{lem}\label{lem:transitivity}
	The action of ${\rm Aut}(T^2)$ on $\mathcal{C}^{i=2}(T^2)$ is transitive on both the vertex set and the edge set.
\end{lem}

\begin{proof}
	For any two coprime numbers $p,q$, we can take integers $r,s$ such that $ps-qr=1$.
	Then an automorphism $f$ of $T^2$ with matrix
	\begin{displaymath}
		A_f=\left(
		\begin{matrix}
			r & p \\
			s & q \\
		\end{matrix}
		\right)^{-1}
		=\left(
		\begin{matrix}
			-q & p \\
			s & -r
		\end{matrix}
		\right)
	\end{displaymath}
	takes the $p/q$-curve class to the $0/1$-curve class.
	So the action of ${\rm Aut}(T^2)$ is transitive on the vertex set.
	
	For the transitivity on the edge set, it suffices to show that for any edge $\{p/q,0/1\}$, there exists an $h\in{\rm Aut}(T^2)$ taking it to $\{2/1,0/1\}$. 
	Since $i(p/q,0/1)=|p|=2$, we may assume that $p=2$ and $q$ is odd.
	Let $f'$ be an automorphism of $T^2$ with matrix
	\begin{displaymath}
		A_{f'}=\left(
		\begin{matrix}
			1 & 0 \\
			\frac{1-q}{2} & 1
		\end{matrix}
		\right).
	\end{displaymath}
	Then we have $f'_*(0/1)=0/1$ and $f'_*(p/q)=2/1$.
\end{proof}

\begin{rem}
	Two vertices $p/q,p'/q'$ have geometric intersection number 2 if and only if in the Farey tessellation (Figure \ref{fig:Farey}) there exists an edge $e$ such that $p/q,p'/q'$ and $e$ determine two adjacent ideal triangles.
	In fact, if $i(p/q,p'/q')=2$, then by transitivity of the ${\rm Aut}(T^2)$-action on both $\mathcal{C}^{i=1}(T^2)$ and $\mathcal{C}^{i=2}(T^2)$, we may assume that $p/q=1/1,\,p'/q'=-1/1$ and thus such an edge $e$ exists: the one connecting $1/0$ and $0/1$; conversely, if such an edge $e$ exists, again by transitivity we may assume that $e$ is the one connecting $1/0,0/1$ and thus $p/q,p'/q'$ can only be $1/1,-1/1$.
\end{rem}

The distance between two vertices $p/q$ and $p'/q'$ in $\mathcal{C}^{i=2}(T^2)$ is the minimum number of edges in a path connecting $p/q$ and $p'/q'$, denoted by $d(p/q,p'/q')$.
If there is no path between them, i.e., $p\not\equiv p'$ or $q\not\equiv q' \bmod 2$, then let $d(p/q,p'/q')$ be $\infty$.
Now ${\rm Aut}(T^2)$ and $GL(2,\mathbb{Z})$ act on $\mathcal{C}^{i=2}(T^2)$ by isometries.

Suppose that $p,q\in\mathbb{Z}$ are nonzero and coprime to each other.
Moreover, assume that $p$ is even.
Take a continued fraction of $|p|/|q|$:
\begin{displaymath}
	\frac{|p|}{|q|} =[a_0,a_1,\cdots,a_n]
	=a_0+\cfrac{1}{a_1+\cfrac{1}{\ddots+\cfrac{1}{a_n}}},
\end{displaymath}
where $a_0,a_1,\cdots,a_n$ are integers with $a_0\geq 0$, $a_1,\cdots,a_{n-1}>0$, $a_n>1$.
Let $b_0=a_0$ and for $i\geq 1$ inductively define
\begin{displaymath}
	b_i=\begin{cases}
		a_i, &\text{if }b_{i-1}\neq a_{i-1}\text{ or }\sum\limits_{j=0}^{i-1}b_j\text{ is odd};\\
		0, &\text{if }b_{i-1}=a_{i-1}\text{ and }\sum\limits_{j=0}^{i-1}b_j\text{ is even}.
	\end{cases}
\end{displaymath}
Let 
\begin{displaymath}
	N(p,q)=\frac{1}{2}\sum_{i=0}^{n}b_i.
\end{displaymath}
This was defined by Bredon and Wood \cite{BW1969}. 
They gave the following theorem.

\begin{thm}[{\cite[Theorem 6.1 \& \S 3]{BW1969}}]
	\label{thm:BW1969}
	Suppose that $p,q\in\mathbb{Z}$ are nonzero and coprime to each other.
	A non-orientable closed surface of genus $g$ can be embedded in the lens space $L(p,q)$ if and only if $p$ is even and $g$ equals $N(p,q)+2n$ with $n\in\mathbb{Z}$, $n\geq 0$.
	Particularly, if $p$ is even, $N(p,q)$ is an invariant of $L(p,q)$ with $N(p,q) \equiv p/2 \bmod 2$.
\end{thm}

For convenience, let $N(p,q)$ be $\infty$ if $p$ is odd, and let $N(0,\pm 1)=0$.
Then the distance between two vertices in $\mathcal{C}^{i=2}(T^2)$ can be determined as follows.

\begin{lem}[{\cite[Corollary 2.6]{Przytycki2019}}]
	\label{lem:distance}
	For any vertex $p/q\in\mathbb{Q}\cup\{\infty\}$, $d(0/1,p/q)$ equals $N(p,q)$.
\end{lem}

\begin{cor}
	\label{cor:distance}
	For two vertices $p/q,p'/q'\in\mathbb{Q}\cup\{\infty\}$, the distance $d(p/q,p'/q')$ equals $N(pq'-qp',p's-q'r)$, where $r,s$ are integers satisfying $ps-qr=1$.
	In particular, $d(1/0,p/q)=N(q,p)$, and $d(1/1,p/q)=N(q-p,p)$.
\end{cor}

\begin{proof}
	If $p\not\equiv p'$ or $q\not\equiv q'\bmod 2$, then $pq'\not\equiv qp'\bmod 2$, and hence $N(pq'-qp',p's-q'r)=\infty=d(p/q,p'/q')$.
	
	If $p\equiv p'$ and $q\equiv q'\bmod 2$, let $f$ be an automorphism of $T^2$ with matrix
	\begin{displaymath}
		A_f=\left(
		\begin{matrix}
			-q & p \\
			s & -r \\
		\end{matrix}
		\right).
	\end{displaymath}
	Then $f(p/q)=0/1$ and $f(p'/q')=(pq'-qp')/(p's-q'r)$.
	So 
	\begin{displaymath}
		d(p/q,p'/q')=d(f(p/q),f(p'/q'))=N(pq'-qp',p's-q'r). \qedhere
	\end{displaymath}
\end{proof}

\section{Incompressible surfaces}
\label{sect:incompressible-surface}

\subsection{Basic concepts}

Here are some basic concepts in the study of embedded surfaces in 3-manifolds, which can be found in 
\cite[pp. 5, 13, 16, 18 -- 19]{Hatcher2023}.
Suppose that $M$ is a 3-manifold and $F\subset M$ is a compact surface (not necessarily connected or orientable) which is \textbf{properly embedded}, i.e., satisfying $F\cap \partial M=\partial F$.
\begin{itemize}
	\item A \textbf{compressing disk} for $F$ is an embedded disk $D$ in $M$ with $D\cap F=\partial D$.
	Moreover, if $\partial D$ can not bound a disk in $F$, then $D$ is called a \textbf{nontrivial} compressing disk.
	If $F$ has no nontrivial compressing disk, and each component of $F$ is neither a sphere nor a disk, then $F$ is said to be \textbf{geometrically incompressible}, or briefly \textbf{incompressible} in this paper.
	If $F$ has a nontrivial compressing disk $D$, then one can \textbf{compress $F$ along $D$}, i.e., take a closed tubular neighborhood $N(D)\cong D\times [-1,1]$ of $D$ with $N(D)\cap F=\partial D\times[-1,1]$, and then get a new surface $(F-N(D))\cup(D\times\{\pm 1\})$.
	
	\item A \textbf{$\partial$-compressing disk} for $F$ is an embedded disk $D$ in $M$ such that $\partial D \cap F$ is an arc $\alpha$ in $\partial D$ and $\partial D-\alpha \subset \partial M$.
	Moreover, if there is no embedded disk $D' \subset F$ with $\alpha \subset \partial D'$ and $\partial D' - \alpha \subset \partial M$, then $D$ is called a \textbf{nontrivial} $\partial$-compressing disk.
	If $F$ has no nontrivial $\partial$-compressing disk, then $F$ is said to be \textbf{$\partial$-incompressible}.
	If $F$ has a nontrivial $\partial$-compressing disk $D$, then one can \textbf{$\partial$-compress $F$ along $D$}, i.e., take a closed tubular neighborhood $N(D)\cong D\times [-1,1]$ of $D$ with $N(D)\cap F=\alpha\times[-1,1]$, $N(D)\cap \partial M=\overline{\partial D-\alpha}\times [-1,1]$, and then get a new surface $(F-N(D))\cup(D\times\{\pm 1\})$.
\end{itemize}
Moreover, $F$ is \textbf{$\partial$-parallel} if it is isotopic, fixing $\partial F$, to a subsurface of $\partial M$, or equivalently, $F$ splits off an embedded $F \times I$ in $M$ with $F=F\times\{0\}$ (in this case $F$ is said to be \textbf{parallel} to the subsurface $\partial F\times I\cup F\times\{1\}$ of $\partial M$).
If $F$ is either incompressible and $\partial$-incompressible, or a sphere not bounding an embedded 3-ball in $M$, or a disk that is not $\partial$-parallel, then we call $F$ an \textbf{essential} surface.

\begin{lem}
	\label{lem:component}
	For a properly embedded surface $F$ in a 3-manifold $M$, $F$ is incompressible if and only if each component of $F$ is incompressible; $F$ is $\partial$-incompressible if and only if each component is $\partial$-incompressible; $F$ is essential if and only if each component is essential.
\end{lem}

\begin{proof}
    We only give a proof for incompressibility. 
	For $\partial$-incompressibility and essentiality, the proofs are similar.
    
	One direction is obvious: if each component of $F$ is incompressible, then $F$ has no nontrivial compressing disk either and thus is also incompressible.
	For the other direction we use proof by contradiction. 
	Assume that $F=F_1\cup F_2\cup \cdots \cup F_n$ is incompressible but one component $F_1$ has a nontrivial compressing disk $D_1$ in $M$.
	Then $D_1$ must intersect $F_2\cup\cdots\cup F_n$, for otherwise $D_1$ becomes a nontrivial compressing disk for $F$.
	We can choose $D_1$ so that the number of components of $D_1\cap(F_2\cup\cdots\cup F_n)$ is minimized.
	Take an innermost component $C_0$ of $D_1\cap(F_2\cup\cdots\cup F_n)$, which bounds a disk $D_0\subset D_1$ with ${\rm Int}(D_0)\cap F=\emptyset$.
	Since $F$ is incompressible, $C_0$ must bound a disk $D_0'$ in $F$.
	Then we can replace $D_1$ by $(D_1-D_0)\cup D_0'$ with a little perturbation to eliminate $C_0$ from $D_1\cap(F_2\cup\cdots\cup F_n)$.
	This is a contradiction to our assumption that the number of its components is minimized.
\end{proof}

\subsection{Incompressible surfaces in the thickened torus and the solid torus}
\label{subsect:torus}

Incompressible surfaces in the thickened torus $T^2\times I$ are classified by Przytycki:

\begin{thm}[{\cite[Theorem 2.3]{Przytycki2019}}]
	\label{thm:thickened-torus}
	If $F$ is an incompressible connected surface in $T^2\times I$, then it is isotopic to one of the following surfaces:
	\begin{enumerate}
		\item a $\partial$-parallel torus;
		\item a $\partial$-parallel annulus;
		\item a vertical annulus $c\times I$ with $c$ an essential simple closed curve on $T^2$;
		\item a non-orientable surface bounded by two essential simple closed curves $c_0\subset T^2\times\{0\}$, $c_1\subset T^2\times\{1\}$ of different slopes $p_0/q_0, p_1/q_1$, whose genus equals $d(p_0/q_0,p_1/q_1)$.
	\end{enumerate}
\end{thm}

Hereafter we use $\Pi_{g,b}$ to denote the non-orientable, connected, compact surface of genus $g$ and with $b$ boundary components. 
Moreover, $\Pi_{g,0}$ is denoted by $\Pi_g$ briefly. 

The non-orientable surfaces in Theorem \ref{thm:thickened-torus} (4) can be constructed as follows.
First for a $1/1$-curve on $T^2\times\{0\}$ and a $-1/1$-curve on $T^2\times\{1\}$, there is an embedded $\Pi_{1,2}$ in $T^2\times I$ taking them as boundary components.
See Figure \ref{fig:thickened-torus}, where $T^2\times I$ is obtained from the cube by gluing the front face and the back face, as well as the left face and the right face, in a trivial way.
By Lemma \ref{lem:transitivity}, for each edge $\{p/q,p'/q'\}$ in $\mathcal{C}^{i=2}(T^2)$, a $p/q$-curve on $T^2\times\{0\}$ and a $p'/q'$-curve on $T^2\times\{1\}$ cobound an embedded $\Pi_{1,2}$.
Thus given any path connecting $p_0/q_0,p_1/q_1$ in $\mathcal{C}^{i=2}(T^2)$ of length $l$, we can construct an embedded $\Pi_{1,2}$ in each $T^2\times [\frac{j-1}{l},\frac{j}{l}]$ for the $j$-th edge ($j=1,2,\cdots,l$).
Their union is an embedded $\Pi_{l,2}$ in $T^2\times I$, whose boundary consists of a $p_0/q_0$-curve on $T^2\times\{0\}$ and a $p_1/q_1$-curve on $T^2\times\{1\}$.

\begin{figure}[htbp]
	\centering
	\begin{picture}(110, 105)(0,0)
		\put(0,0){\includegraphics{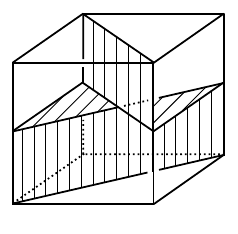}
		}
		\put(50,17){\footnotesize $1/1$}
		\put(56,94){\footnotesize $-1/1$}
	\end{picture}
	\caption{An embedded $\Pi_{1,2}$ in $T^2 \times I$.}
	\label{fig:thickened-torus}
\end{figure}

Suppose that $F$ is an incompressible connected surface in the solid torus $D^2 \times S^1$. 
If $F$ is orientable, then according to \cite[Lemmas 1.10 \& 1.11]{Hatcher2023}, it is a $\partial$-parallel annulus.
So by \cite[\S 3]{Rubinstein1978} (also see the explanation in \cite[p. 107]{Frohman1986}), \cite[Corollary 2.2.]{Frohman1986}, and \cite[Corollary 2.5]{Przytycki2019}, we have the following conclusion on the structure of incompressible surfaces in the solid torus.
Here we fix an orientation of $D^2\times S^1$, and a simple closed curve $c$ on its boundary is called a $p/q$-curve if with an orientation $[c]$ equals 
\begin{displaymath}
	p[*\times S^1]+q[\partial D^2 \times *] \in H_1(\partial D^2 \times S^1; \mathbb{Z}) \cong \mathbb{Z} \langle [*\times S^1] \rangle \oplus \mathbb{Z} \langle [\partial D^2 \times *] \rangle.
\end{displaymath}

\begin{prop}
	\label{prop:solid-torus}
	Suppose that $F$ is a connected surface and properly embedded in the solid torus $D^2 \times S^1$. 
	Then $F$ is incompressible in $D^2\times S^1$ if and only if either $F$ is a $\partial$-parallel annulus, or $\partial F$ is a $2k/q$-curve on $\partial D^2 \times S^1$ with $k,q \in \mathbb{Z}\backslash\{0\}$ and $F\cong \Pi_{N(2k,q),1}$.
	Moreover, $F$ is essential if and only if $F$ is isotopic to a meridian disk $D^2\times *$.
\end{prop}

According to \cite[\S 3]{Rubinstein1978}, for a given $2k/q$-curve on $\partial D^2 \times S^1$, the incompressible surface $\Pi_{N(2k,q),1}$ bounded by it in $D^2\times S^1$ is unique up to isotopy.
It can be constructed as follows.
First take a smaller disk $D_0^2$ in $D^2$.
By Theorem \ref{thm:thickened-torus}, in $(D^2-{\rm Int}(D_0^2))\times S^1 \cong T^2\times I$ there is a non-orientable surface of genus $d(0/1,2k/q)=N(2k,q)$ bounded by $\partial D_0^2\times *$ and a $2k/q$-curve on $\partial D^2 \times S^1$.
So the union of this surface and $D_0^2\times *$ is a $\Pi_{N(2k,q),1}$.
Moreover, this implies that the non-orientable surfaces in Theorem \ref{thm:thickened-torus} (4) are all incompressible in the thickened torus.

\begin{lem}[{\cite[Lemma 2.8]{Du2022}}]
	\label{lem:connected}
	If $F_1$ and $F_2$ are two incompressible non-orientable surfaces in the solid torus, then $F_1\cap F_2\neq \emptyset$.
\end{lem}

\begin{cor}
	\label{cor:thickened-torus}
	If $F$ is an incompressible surface in $T^2\times I$, then up to isotopy it is one of the following surfaces:
	\begin{enumerate}
		\item a union of several $\partial$-parallel annuli and several $\partial$-parallel tori;
		\item a union of several $\partial$-parallel annuli and several parallel copies of a vertical annulus $c\times I$ with $c$ an essential simple closed curve on $T^2$;
		\item a union of several $\partial$-parallel annuli and a non-orientable surface bounded by two essential simple closed curves $c_0\subset T^2\times\{0\}$, $c_1\subset T^2\times\{1\}$ of different slopes $p_0/q_0, p_1/q_1$, whose genus equals $d(p_0/q_0,p_1/q_1)$.
	\end{enumerate}
\end{cor}

\begin{proof}
	By Lemma \ref{lem:component}, each component of $F$ is isotopic to one of the surfaces in Theorem \ref{thm:thickened-torus}.
	If there are two non-orientable components $F_1$, $F_2$ in $T^2\times I$, then we can glue a solid torus $D^2\times S^1$ to $T^2\times I$ via a homeomorphism $\partial D^2\times S^1\xrightarrow{\cong}T^2\times\{0\}$, such that $(\partial F_1\cup \partial F_2) \cap (T^2\times\{0\})$ bounds a union of meridian disks in $D^2\times S^1$.
	Then we get a solid torus, and the union of $F_1\cup F_2$ and these meridian disks gives two incompressible non-orientable surfaces, a contradiction to Lemma \ref{lem:connected}.
	Therefore, there is at most one non-orientable component in $F$ and the conclusion follows.
\end{proof}

\subsection{Incompressible surfaces in Seifert fibered manifolds}

A 3-manifold $M$ is \textbf{irreducible}, if each embedded 2-sphere bounds an embedded 3-ball.
The structure of incompressible surfaces and essential surfaces in irreducible Seifert fibered 3-manifolds was investigated in \cite{Waldhausen1967,Hatcher2023,Frohman1986,Rannard1996,KN2023}.
For convenience here we only give an introduction to the result for compact orientable Seifert fibered 3-manifolds with orientable base surfaces.
Such a 3-manifold can be constructed as follows.
Denote by $\Sigma_{g,b}$ a compact orientable surface of genus $g$ and with $b$ boundary components.
Let $B=\Sigma_{g,b}$ with $b\geq 0$ (the base surface).
Take $n$ disjoint closed disks $B_1,B_2,\cdots,B_n$ in its interior ${\rm Int}(B)$ (we may always assume that $n\geq 1$). 
Let $B_0$ be $B-\cup_{i=1}^n{\rm Int}(B_i)\cong\Sigma_{g,b+n}$ and let $M_0$ be $B_0\times S^1$.
Fix orientations of $B_0$ and $S^1$, so $M_0$ and $\partial B_1,\partial B_2\cdots,\partial B_n$ are all oriented.
For each $i=1,2,\cdots,n$, take a solid torus $V_i=D_i^2\times S_i^1$ (with a suitable orientation) and attach it to $M_0$ 
via $f_i:\partial D_i^2\times S_i^1\xrightarrow{\cong}\partial B_i\times S^1$, which induces
\begin{displaymath}
	\begin{matrix}
		(f_i)_*: & H_1(\partial D_i^2\times S_i^1;\mathbb{Z}) & \xrightarrow{\cong} & H_1(\partial B_i\times S^1;\mathbb{Z}) \\
		&  \left(\begin{matrix}
			[\partial D_i^2\times *] & [*\times S_i^1]
		\end{matrix}\right)
		&\mapsto &
		\left(\begin{matrix}
			[\partial B_i\times *] & [*\times S^1]
		\end{matrix}\right)
		\left(\begin{matrix}
			\alpha_i & \gamma_i \\
			\beta_i & \delta_i
		\end{matrix}\right)
	\end{matrix}
\end{displaymath}
with $\alpha_i,\beta_i,\delta_i,\gamma_i\in\mathbb{Z}$, $\alpha_i\delta_i-\beta_i\gamma_i=1$, and $\alpha_i \neq 0$.
Then we get a compact orientable 3-manifold $M_0\cup V_1\cup\cdots\cup V_n$. 
Using the notation in \cite{Du2022}, we denote it by
\begin{displaymath}
	M=(\Sigma_{g,b},(\alpha_1,\beta_1),\cdots,(\alpha_n,\beta_n)).
\end{displaymath}
Note that $M_0=B_0\times S^1$ is fibered by $S^1$ naturally and this fibering extends to each $V_i$ as $\alpha_i\neq 0$.
So $M$ is Seifert fibered. 
For each $i=1,2,\cdots,n$, the fiber $\{0\}\times S_i^1\subset D_i^2\times S_i^1$ is singular if and only if $\alpha_i\neq \pm 1$. 
Other fibers are all regular.

For a properly embedded surface $F$ in $M$, it is called 
\begin{itemize}
	\item \textbf{vertical}, if $F$ is a union of regular fibers of the Seifert fibered manifold;
	\item \textbf{pseudo-vertical}, if $F\cap M_0$ is vertical in $M_0$ and for each $i=1,2,\cdots,n$, $F\cap V_i$ is either empty or an incompressible non-orientable surface;
	\item \textbf{horizontal}, if $F$ is transverse to all fibers;
	\item \textbf{pseudo-horizontal}, if $F\cap M_0$ is horizontal in $M_0$ and for each $i=1,2,\cdots,n$, $F\cap V_i$ is either an incompressible non-orientable surface or a collection of meridian disks of $V_i$.
\end{itemize}

\begin{rem}
	Using the notation in \cite[\S 2.1]{Hatcher2023}, such a Seifert fibered 3-manifold is denoted by
	\begin{displaymath}
	    M(g,b;\beta_1/\alpha_1,\beta_2/\alpha_2,\cdots,\beta_n/\alpha_n).
	\end{displaymath}
	If $M$ is closed, i.e., $b=0$, the sum $\sum\limits_{i=1}^n\beta_i/\alpha_i$ is an invariant, called the \textbf{Euler number} of the Seifert fibering, and horizontal surfaces exist if and only if the Euler number is zero \cite[Proposition 2.2]{Hatcher2023}.
\end{rem}

\begin{thm}[{\cite[Theorem 1.4]{KN2023}}]
	\label{thm:essential}
	Suppose that $M=(\Sigma_{g,b},(\alpha_1,\beta_1),\cdots,(\alpha_n,\beta_n))$ is irreducible and has at least one singular fiber.
	Then each essential connected surface in $M$ is isotopic to a pseudo-vertical surface or a pseudo-horizontal surface. 
\end{thm}

\begin{thm}[{\cite[Lemma 2.16]{Jackson2025}}]
	\label{thm:pseudo-horizontal}
	Each essential surface in $(\Sigma_{g,b},(\alpha_1,\beta_1),\cdots,(\alpha_n,\beta_n))$ with $b\geq 1$ is isotopic to a pseudo-vertical surface or a horizontal surface.
\end{thm}

For a pseudo-vertical surface $F$, if $F\cap V_i\neq\emptyset$, then by Proposition \ref{prop:solid-torus} and Lemma \ref{lem:connected}, it is an incompressible, connected, non-orientable surface in $V_i$ whose boundary is a $2k/q$-curve with $k,q\in\mathbb{Z}\backslash\{0\}$ and $2k,q$ coprime to each other.
Since $F\cap\partial V_i$ is identified with a vertical circle in $M_0=B_0\times S^1$, we see that $2k/q=-\alpha_i/\gamma_i$.
So $\alpha_i$ must be even.
Since $-\beta_i\gamma_i\equiv 1\bmod\alpha_i$, we have $L(\alpha_i,-\gamma_i)\cong L(\alpha_i,\beta_i)$, so the genus of $F\cap V_i$ equals $N(2k,q)=N(\alpha_i,-\gamma_i)=N(\alpha_i,\beta_i)$.
Generally we have the following classification.

\begin{lem}
    \label{lem:pseudo-vertical}
	Suppose that $M=(\Sigma_{g,b},(\alpha_1,\beta_1),\cdots,(\alpha_n,\beta_n))$ is constructed as above and $F$ is a pseudo-vertical connected surface in it.
	Then $F$ is one of the following surfaces:
	\begin{enumerate}
		\item a vertical torus or a vertical annuls in $M_0=B_0\times S^1$;
		\item a pseudo vertical closed non-orientable surface of genus $N(\alpha_i,\beta_i)+N(\alpha_j,\beta_j)$ ($i,j\in\{1,2,\cdots,n\}$, $\alpha_i,\alpha_j\in 2\mathbb{Z}\backslash\{0\}$), with $F\cap M_0\cong S^1\times I$, $F\cap V_k \cong \Pi_{N(\alpha_k,\beta_k),1}$ for $k=i,j$ and $F\cap V_l=\emptyset$ for $l\neq i,j$;
		\item a pseudo vertical non-orientable surface with connected boundary, of genus $N(\alpha_i,\beta_i)$ ($i\in\{1,2,\cdots,n\}$, $\alpha_i\in 2\mathbb{Z}\backslash\{0\}$), with $F\cap M_0\cong S^1\times I$, $F\cap V_i \cong \Pi_{N(\alpha_i,\beta_i),1}$ and $F\cap V_j=\emptyset$ for $j\neq i$.
	\end{enumerate}
\end{lem}

\subsection{Essential surfaces in the twisted \texorpdfstring{$I$}{I}-bundle over the Klein bottle}
\label{subsect:essential}

Take a thickened torus $S^1\times S^1\times [-1,1]$ with coordinate $(z,w,t)$, $z,w\in S^1\subset\mathbb{C}$, $t\in [-1,1]$.
Let $\sigma$ be the involution of $S^1\times S^1\times [-1,1]$ defined by $\sigma(z,w,t)=(-z,\overline{w},-t)$.
Denote the quotient space $S^1\times S^1\times [-1,1]/\sigma$ by $N$, and the projection by $\rho: S^1\times S^1\times I\to N$. 
Figure \ref{fig:semi-bundle} shows a fundamental domain 
$$\{(z,w,t)\in S^1\times S^1\times I: {\rm Im}\,z\geq 0\}$$
for the $\sigma$-action, and $N$ can be obtained by gluing its left side with the right side according to the arrows.
We see that $N$ is an orientable 3-manifold and can be viewed as a twisted $I$-bundle over the Klein bottle $K=\rho(S^1\times S^1\times \{0\})$.
The boundary of $N$ is a torus $\rho(S^1\times S^1\times \{\pm 1\})$. 
We can take two embedded circles on it: $\lambda=\rho(S^1\times \{1\}\times \{\pm 1\})$, $\mu=\rho(\{1\}\times S^1\times \{1\})$.
Moreover, let $\alpha$ be the circle $\rho(S^1\times \{1\}\times \{0\})$.
For each of these circles, we fix an orientation.

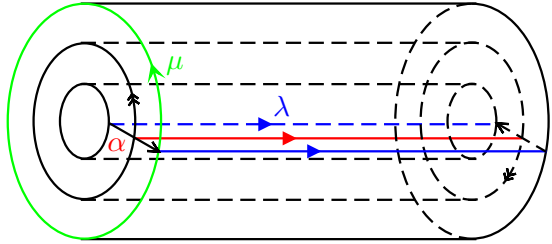
\begin{figure}[htbp]
	\centering
	\begin{picture}(210, 100)(0,0)
		\put(0,0){\includegraphics{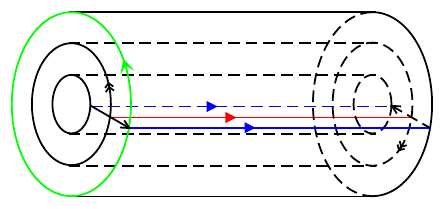}}
		\put(105,52){\color{blue} $\lambda$}
		\put(65,70){\color{green} $\mu$}
		\put(43,39){\color{red} $\alpha$}
	\end{picture}
	\caption{\label{fig:semi-bundle}
		A fundamental domain for the $\sigma$-action on $S^1\times S^1\times I$.}
\end{figure}

Note that $N$ can be Seifert fibered as $(\Sigma_{0,1},1/2,-1/2)$, where the two singular fibers are $\alpha=\rho(S^1\times \{1\}\times \{0\})$ and $\rho(S^1\times \{-1\}\times \{0\})$.
According to \cite[Proposition 1.13]{Hatcher2023}, it is irreducible. 
We can classify all essential surfaces in $N$ as follows (cf. the second paragraph in the proof of \cite[Lemma 14]{JRT2013}).

\begin{thm}
	\label{thm:twisted-bundle}
	Suppose that $F$ is a connected surface in $N=(\Sigma_{0,1},1/2,-1/2)$. 
    Then $F$ is essential if and only if it is isotopic to one of the following surfaces (see Figure \ref{fig:essential}):
	
	\begin{enumerate}
		\item a horizontal annulus, isotopic to $\rho(\{{\rm i}\}\times S^1\times[-1,1])$;
		
		\item a vertical annulus, isotopic to $\rho(S^1\times\{\pm{\rm i}\}\times[-1,1])$;
		
		\item a vertical torus, parallel to $\partial N=\rho(S^1\times S^1\times\{\pm 1\})$;
		
		\item a pseudo-vertical M\"obius band, isotopic to $\rho(S^1\times\{1\}\times[-1,1])$ or $\rho(S^1\times\{-1\}\times[-1,1])$;
		
		\item a pseudo-vertical Klein bottle, isotopic to $K=\rho(S^1\times S^1\times \{0\})$.
	\end{enumerate}
\end{thm}

\begin{figure}[htbp]
	\centering
	\begin{picture}(440, 360)(0,0)
		\put(0,0){\includegraphics{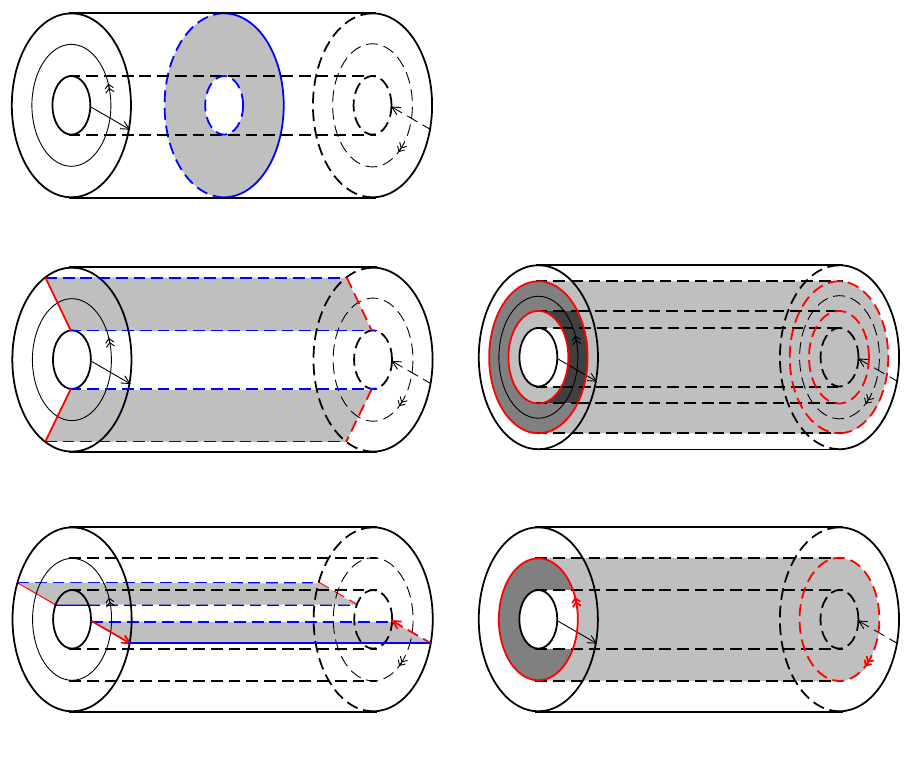}}
		\put(100,253){(1)}
		\put(100,132){(2)}
		\put(100,8){(4)}
		\put(328,132){(3)}
		\put(328,8){(5)}
	\end{picture}
	\caption{\label{fig:essential}
		Essential surfaces in a torus semi-bundle.}
\end{figure}
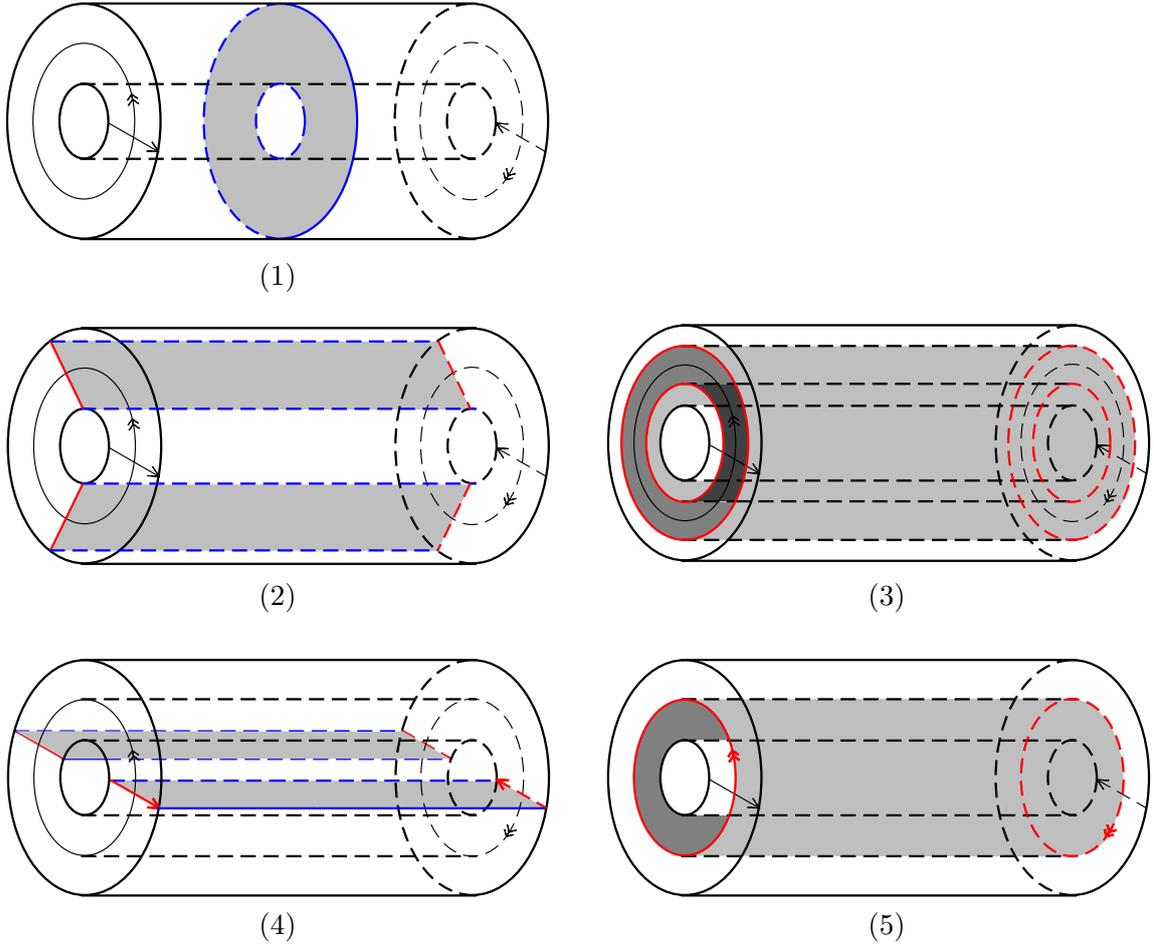

\begin{proof}
    According to \cite[Lemma 2.2]{SWW} (also see \cite[p. 9]{GMT2025}), if $F$ is orientable, then it is essential if and only if it is isotopic to one of (1), (2), (3).
    So we only need to consider non-orientable surfaces.

    If $F$ is horizontal, the projection from $N$ to the base surface $\Sigma_{0,1}$ induces a branched covering $F\to\Sigma_{0,1}$ and thus $F$ is orientable.
	Therefore, by Theorem \ref{thm:pseudo-horizontal}, we may assume that $F$ is pseudo-vertical and non-orientable.
    Then according to Lemma \ref{lem:pseudo-vertical}, $F$ is a M\"obius band or a Klein bottle.
	For the two singular fibers $\rho(S^1\times\{\pm 1\}\times\{1\})$, take open regular neighborhoods $V_1,V_2$.
	Let $N_0$ be $N-V_1-V_2$, fibered as $(\Sigma_{0,3};\,\,)=\Sigma_{0,3}\times S^1$.
    Then $F\cap N_0$ is a vertical annulus in $N_0$, and the projection of $F\cap N_0$ on $\Sigma_{0,3}$ is an arc with two endpoints on different boundary components.
	Up to isotopy, there are three such arcs.
    Since $F_0$ is determined by this arc and $F_0\cap V_i$ ($i=1,2$) is determined by $\partial F_0$ up to isotopy, we see that $F$ is isotopic to one of the surfaces in Figure \ref{fig:essential} (4) (5).
	
	To finish the proof, it suffices to show that the surfaces in (4) and (5) are all essential.
	Assume that $F$ is such a surface and has a nontrivial ($\partial$-)compressing disk $D$ in $N$.
	If we ($\partial$-)compress $F$ along $D$ to get another surface $F'$, we will have $\chi(F')>\chi(F)=0$. 
    Note that there is no embedded projective plane in $N$ for it must be incompressible.
    Thus each component of $F'$ can only be either a disk or a sphere.
	Double the pair $(N,F)$, i.e., take two copies $(N_1,F_1),(N_2,F_2)$ and glue them along boundary via the trivial homeomorphism, to obtain $(\widetilde{N},\widetilde{F})$.
	Also, there are two copies $D_1,D_2$ of $D$ in $\widetilde{N}$, and we can compress $\widetilde{F}$ along the disk(s) $D_1\cup D_2$. 
	The result surface is a union of two copies of $F'$, which can only be a union of spheres.
	Since the compressing does not change the surface class in $H_2(\widetilde{N};\mathbb{Z}_2)$ and $\widetilde{N}$ is irreducible by \cite[Proposition 1.13]{Hatcher2023}, we have $[\widetilde{F}]=0\in H_2(\widetilde{N};\mathbb{Z}_2)$.
	If $F$ is a M\"obius band in (4) (or the Klein bottle $K$ in (5), resp.), then comparing Figure \ref{fig:essential} with Figure \ref{fig:semi-bundle}, we see that the curve $\mu$ (or  $\alpha$ with a little perturbation, resp.) intersects $F$ at exactly one point. 
    So there is a simple closed curve in $\widetilde{N}$ intersecting $\widetilde{F}$ at exactly one point, which implies $[\widetilde{F}]\neq 0\in H_2(\widetilde{N};\mathbb{Z}_2)$, a contradiction.
\end{proof}


\section{\texorpdfstring{$\mathbb{Z}_2$}{Z2}-Thurston norm}
\label{sect:norm}

In this section we determine $\mathbb{Z}_2$-Thurston norms of all $\mathbb{Z}_2$-homology classes and construct $\mathbb{Z}_2$-taut surfaces realizing them for the following two kinds of 3-manifolds: torus bundles over $S^1$ (including non-orientable ones), and torus semi-bundles.
In particular, $\mathbb{Z}_2$-Thurston norms in Sol manifolds are all determined.

Note that if $M$ is an irreducible closed 3-manifold, then each nontrivial element $\alpha$ in $H_2(M;\mathbb{Z}_2)$ can be represented by an incompressible surface $F$.
Moreover, if there is no embedded projective plane in $M$, then by definition $F$ can be chosen as a $\mathbb{Z}_2$-taut surface.

\subsection{Torus bundles over \texorpdfstring{$S^1$}{S1}}

For $f\in{\rm Aut}(T^2)$, the \textbf{mapping torus} of $f$, denoted by $M_f$, is the manifold $T^2\times I/\sim_f$, where $\sim_f$ is defined by $(0,x)\sim_f(1,f(x))$ for each $x\in T^2$.
It is orientable if and only if $f$ is orientation preserving.
A 3-manifold is a torus bundle over $S^1$ if and only if $M=M_f$ for some $f\in{\rm Aut}(T^2)$.
Up to homeomorphism, the mapping torus $M_f$ of $f$ is determined by the matrix $A_f$:

\begin{thm}[{\cite[Theorem 2.6]{Hatcher2023}}]
	For $f,f'\in{\rm Aut}(T^2)$, the mapping tori $M_f,M_{f'}$ of $f,f'$ are homeomorphic if and only if $A_f$ is conjugate to $A_{f'}$ or $A_{f'}^{-1}$ in $GL(2,\mathbb{Z})$.
\end{thm}

The universal covering space of $M_f$ with $f\in{\rm Aut}(T^2)$ is homeomorphic to $\mathbb{R}^3$, and hence by \cite[Proposition 1.6]{Hatcher2023}, $M_f$ is irreducible.

\begin{lem}
	\label{lem:fiber}
	Suppose that $F$ is an incompressible surface in $M_f$ with $f\in{\rm Aut}(T^2)$.
	Then we can choose a torus fiber $T_0$ of $M_f$ such that $F\cap T_0$ is either empty or a union of several essential simple closed curves on $T_0$.
\end{lem}

\begin{proof}
	After an isotopy, we may assume that $F\cap T_0$ consists of $m$ simple closed curves with $m\geq 0$ minimized.
	If a component $c$ of $F\cap T_0$ bounds a disk $D$ in $T_0$, we may assume that $c$ is innermost, i.e., ${\rm Int}(D)$ contains no other component of $F\cap T_0$.
	Since $F$ is incompressible, $c$ also bounds a disk $D'$ in $F$.
	As $M_f$ is irreducible, the sphere $D\cup D'$ bounds a 3-ball $B$ in $M_f$. 
	We can push $F$ along $B$ to eliminate $c$ from $F\cap T_0$ and decrease $m$, a contradiction.
	It follows that each component of $F\cap T_0$ is an essential simple closed curve on $T_0$.
\end{proof}

\begin{cor}
	There is no embedded $\mathbb{R}P^2$ in $M_f$ for any $f\in{\rm Aut}(T^2)$.
\end{cor}

\begin{proof}
	Assume that there is an embedded $\mathbb{R}P^2$ in $M_f$, denoted by $F$.
	Then $F$ is incompressible.
	Take a torus fiber $T_0$ as in the above lemma.
	Cutting $M_f$ along $T_0$, we get a thickened torus $T^2\times I$ and $F$ becomes several connected pieces $F_1,\cdots,F_k$ properly embedded in it.
	However, each $F_i$ is incompressible and thus by Theorem \ref{thm:thickened-torus}, homeomorphic to one of the followings: the torus, the annulus, a non-orientable surface with two boundary components.
	As a consequence, $\chi(F_i)\leq 0$, and
	\begin{displaymath}
		\chi(F)=\sum_{i=1}^k\chi(F_i)\leq 0,
	\end{displaymath}
	a contradiction.
\end{proof}
 
For any given $f\in{\rm Aut}(T^2)$ with 
\begin{displaymath}
	A_f=
    \left(
	\begin{matrix}
		a & c \\
		b & d
	\end{matrix}
	\right) 
    \in GL(2,\mathbb{Z}),
\end{displaymath}
by an isotopy we may assume that $f$ fixes a point $x\in T^2$.
Let $\gamma_x,\gamma_\lambda,\gamma_\mu\in H_1(M_f;\mathbb{Z}_2)$ be represented by the images of $\{x\}\times I,\lambda\times\{0\},\mu\times\{0\}\subset T^2\times I$ in $M_f=T^2\times I/\sim_f$ respectively.
We have
\begin{displaymath}
	H_1(M_f;\mathbb{Z}_2)\cong 
    \mathbb{Z}_2\langle \gamma_x \rangle \oplus \mathbb{Z}_2 \langle \gamma_\lambda \rangle \oplus \mathbb{Z}_2 \langle \gamma_\mu \rangle / 
    \langle \gamma_\lambda-(a\gamma_\lambda+b\gamma_\mu),\gamma_\mu-(c\gamma_\lambda+d\gamma_\mu)\rangle,
\end{displaymath} 
and the modulo 2 intersection pairing gives $H_2(M_f;\mathbb{Z}_2)\cong H_1(M_f;\mathbb{Z}_2)$.
It can be verified that the number of nontrivial elements in 
\begin{displaymath}
	\mathbb{Z}_2 \langle \gamma_\lambda \rangle \oplus \mathbb{Z}_2 \langle \gamma_\mu \rangle/\langle \gamma_\lambda-(a\gamma_\lambda+b\gamma_\mu),\gamma_\mu-(c\gamma_\lambda+d\gamma_\mu)\rangle
\end{displaymath}
equals the number of $f$-invariant components in $\mathcal{C}^{i=2}(T^2)$.
To determine the $\mathbb{Z}_2$-Thurston norm on $H_2(M_f;\mathbb{Z}_2)$, we are to consider the $f$-action on $\mathcal{C}^{i=2}(T^2)$.
Recall that $\mathcal{C}^{i=2}(T^2)$ consists of three trees: $T_{j/k}$ with vertex set $V_{j/k}$, $j/k\in\{1/0,0/1,1/1\}$.
Below we let
\begin{displaymath}
	l_{j/k} = \inf\{d(v,f(v)): v\in V_{j/k}\}.
\end{displaymath}

\begin{lem}[{\cite[\S 6.4]{Serre1980}}]
	Suppose that $\varphi$ is an automorphism of a tree.
	Then one of the following holds:
	\begin{itemize}
		\item $\varphi$ is a \textbf{rotation}, i.e., $\varphi$ preserves a vertex;
		\item $\varphi$ is an \textbf{inversion}, i.e., $\varphi$ preserves an edge and exchanges the two endpoints;
		\item $\varphi$ is a \textbf{translation}, i.e., there is a $\varphi$-invariant straight path on which $\varphi$ induces a translation.
	\end{itemize}
\end{lem}

If $f$ fixes a component $T_{j/k}$ ($j/k\in\{1/0,0/1,1/1\}$), we can construct an embedded closed surface $F_{j/k}$ in $M_f=T^2\times I/\sim_f$ as follows, which intersects $T^2\times\{0\}$ at an essential simple closed curve $c_{j/k}\times\{0\}$ with $[c_{j/k}]\in V_{j/k}$.
\begin{itemize}
	\item If $f$ acts on $T_{j/k}$ as a rotation, up to isotopy there is an $f$-invariant essential simple closed curve $c_{j/k}$ on $T^2$ with $[c_{j/k}]\in V_{j/k}$. 
	Then $c_{j/k} \times I \subset T^2 \times I$ provides an embedded surface in $M_f$, denoted by $F_{j/k}$.
    It is either a torus or a Klein bottle, according to whether $f$ preserves the orientation of $c_{j/k}$.
	
	\item If $f$ acts on $T_{j/k}$ as an inversion or translation, there is an essential simple closed curve $c_{j/k}$ on $T^2$ with $[c_{j/k}]\in V_{j/k}$ and $d([c_{j/k}],[f(c_{j/k})])=l_{j/k}$.
	As shown in Section \ref{subsect:torus}, we can construct an embedded $\Pi_{l_{j/k},2}$ in $T^2\times I$ which is bounded by $(c_{j/k}\times\{0\})\cup(f(c_{j/k})\times\{1\})$.
	This provides an embedded $\Pi_{l_{j/k}+2}$ in $M_f$, denoted by $F_{j/k}$.
\end{itemize}
Denote the modulo 2 intersection number of $\alpha\in H_2(M_f;\mathbb{Z}_2)$ and $\beta\in H_1(M_f;\mathbb{Z}_2)$ by $\alpha\cdot\beta$.
We see that $[F_{j/k}]\cdot \gamma_\lambda=k$ and $[F_{j/k}]\cdot \gamma_\mu=j$.
Particularly, $F_{j/k}$ represents a nontrivial element in $H_2(M_f;\mathbb{Z}_2)$.

\begin{prop}
	\label{prop:Fjk}
	Suppose that $f\in{\rm Aut}(T^2)$ fixes $T_{j/k}$ with $j/k\in\{1/0,0/1,1/1\}$, and $F_{j/k}$ is the embedded surface in $M_f$ constructed as above. 
	Then $||[F_{j/k}]||_{\mathbb{Z}_2}= l_{j/k}$.
\end{prop}

\begin{proof}
	If $f$ acts on $T_{j/k}$ as a rotation, then $F_{j/k}$ is a torus or a Klein bottle and thus $||[F_{j/k}]||_{\mathbb{Z}_2}=0=l_{j/k}$. 
	Below we assume that $f$ acts on $T_{j/k}$ as an inversion or a translation.
	In this case, $F_{j/k}$ is a non-orientable surface of genus $l_{j/k}+2$ and thus $||[F_{j/k}]||_{\mathbb{Z}_2}\leq l_{j/k}$.
	Take a $\mathbb{Z}_2$-taut, incompressible surface $F$ such that $[F]=[F_{j/k}]$.
	By Lemma \ref{lem:fiber}, we may assume that $F$ intersects a torus fiber $T_0$ of $M_f$ at essential simple closed curves $c_1,c_2,\cdots,c_m$ with slope $p/q$.
	Since $mp\equiv[F_{j/k}]\cdot \gamma_\mu=j\bmod 2$ and $mq\equiv[F_{j/k}]\cdot \gamma_\lambda=k\bmod 2$, $m$ must be odd and $p/q\in V_{j/k}$.
	Cut $M_f$ along $T_0$ to obtain a thickened torus $T^2\times I$.
	Then $F$ becomes an incompressible surface in $T^2\times I$, denoted by $P$, with $\partial P$ consisting of $m$ $p/q$-curves on $T^2\times\{0\}$ and $m$ $f(p/q)$-curves on $T^2\times\{1\}$.
    Moreover, we may assume that among all components of $P$ there is no $\partial$-parallel annulus, for otherwise we can eliminate it by pushing $F$ in $M_f$.
    Note that $f(p/q)\neq p/q$.
    By Corollary \ref{cor:thickened-torus}, $P$ can only be a connected non-orientable surface of genus $d(p/q,f(p/q))$, with $m=1$.
    So $F$ is non-orientable and of genus $d(p/q,f(p/q))+2$, which implies $||[F_{j/k}]||_{\mathbb{Z}_2}=d(p/q,f(p/q))\geq l_{j/k}$.
	This completes the proof.
\end{proof}

The values of $l_{1/0},l_{0/1},l_{1/1}$ can be determined as follows.

\begin{prop}
	\label{prop:l}
	Suppose that $f\in{\rm Aut}(T^2)$ preserves $T_{j/k}$ for some $j/k\in\{1/0,0/1,1/1\}$, and $v\in V_{j/k}$ is an arbitrary vertex.
	
	\begin{enumerate}
		\item If $f^2(v)\neq v$, then $l_{j/k}$ equals $d(v,f^2(v))-d(v,f(v))$.
		\item If $f^2(v)=v$ and $d(v,f(v))$ is odd, then $l_{j/k}$ equals $1$.
		\item If $f^2(v)=v$ and $d(v,f(v))$ is even, then $l_{j/k}$ equals $0$.
	\end{enumerate}
\end{prop}

\begin{proof}
	(1) When $f^2(v)\neq v$, $f$ is not an inversion on $T_{j/k}$. 
	If $f$ acts on $T_{j,k}$ as a translation, then there is an $f$-invariant straight path $\gamma$ with $l_{j/k}=d(v',f(v'))$ for any $v'\in\gamma$. 
	Let $v'$ be the projection of $v$ on $\gamma$. 
	Then
	\begin{displaymath}
		\begin{split}
			& d(v,f^2(v))-d(v,f(v))\\
			=& (d(v,v')+d(v',f^2(v'))+d(f^2(v'),f^2(v)))-(d(v,v')+d(v',f(v'))+d(f(v'),f(v)))\\
			=& (d(v,v')+2l_{j/k}+d(v',v))-(d(v,v')+l_{j/k}+d(v',v))\\
			=& l_{j/k}
		\end{split}
	\end{displaymath}
	(see Figure \ref{fig:translation}). 
	If $f$ acts as a rotation with $f^2(v)\neq v$, then it has a unique fixed point on $T_{j/k}$.
    So we have $l_{j/k}=0$ and for any $f^k(v)\neq v$, $d(v,f^k(v))$ equals 2 times the distance between $v$ and the fixed point. 
    Thus $l_{j/k}=d(v,f^2(v))-d(v,f(v))$ still holds.
	
	(2) \& (3) When $f^2(v)=v$, $f$ acts on $T_{j/k}$ either as a rotation with $l_{j/k}=0$ or as an inversion with $l_{j/k}=1$.
	In either case we have $d(v,f(v))\equiv l_{j/k}\bmod 2$ and thus the conclusion follows.
\end{proof}

\begin{figure}[htbp]
	\centering
	\begin{picture}(240, 90)(0,0)
		\put(0,0){\includegraphics{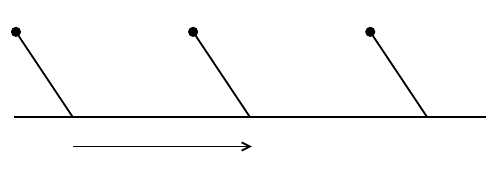}}
		\put(16,75){$v$}
		\put(100,75){$f(v)$}
		\put(185,75){$f^2(v)$}
		\put(38,42){$v'$}
		\put(122,42){$f(v')$}
		\put(207,42){$f^2(v')$}
		\put(73,10){$l_{j,k}$}
		\put(220,25){$\gamma$}
	\end{picture}
	\caption{\label{fig:translation}
		A translation on a tree.}
\end{figure}
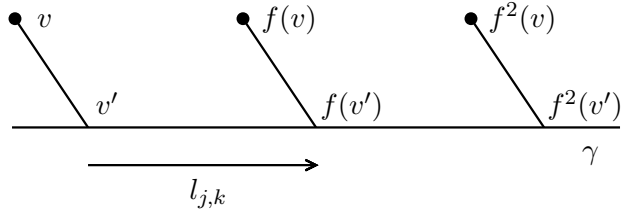

\begin{cor}
	\label{cor:l}
	Suppose that the matrix of $f\in{\rm Aut}(T^2)$ is
	$A_f= \left(
	\begin{matrix}
		a & c \\
		b & d
	\end{matrix}
	\right)\in GL(2,\mathbb{Z})$.
	
	If $a$ is odd and $b$ is even, then
	\begin{displaymath}
		l_{1/0}=
		\begin{cases}
			0, & \text{ if } b(a+d)=0 \text{ and } b/2 \text{ is even};\\
			1, & \text{ if } b(a+d)=0 \text{ and } b/2 \text{ is odd};\\
			N(b(a+d),a^2+bc)-N(b,a), & \text{ if } b(a+d)\neq 0;
		\end{cases}
	\end{displaymath}
	otherwise $l_{1/0}=\infty$.
	
	If $c$ is even and $d$ is odd, then
	\begin{displaymath}
		l_{0/1}=
		\begin{cases}
			0, & \text{ if } c(a+d)=0 \text{ and } c/2 \text{ is even};\\
			1, & \text{ if } c(a+d)=0 \text{ and } c/2 \text{ is odd};\\
			N(c(a+d),bc+d^2)-N(c,d), & \text{ if } c(a+d)\neq 0;
		\end{cases}
	\end{displaymath}
	otherwise $l_{0/1}=\infty$.
	
	If $a+c$ and $b+d$ are both odd, then
	\begin{displaymath}
		l_{1/1}=
		\begin{cases}
			0, &\text{ if } (b-a)(a+c)+(d-c)(b+d)=0 \text{ and } (b+d-a-c)/2 \text{ is even};\\
			
			1, &\text{ if } (b-a)(a+c)+(d-c)(b+d)=0 \text{ and } (b+d-a-c)/2 \text{ is odd};\\
			
			& \hspace{-20pt} N((b-a)(a+c)+(d-c)(b+d),  a(a+c)+c(b+d))-N(b+d-a-c,a+c),\\
			
			&\text{ if } (b-a)(a+c)+(d-c)(b+d) \neq 0;\\
		\end{cases}
	\end{displaymath}
	otherwise $l_{1/1}=\infty$.
\end{cor}

\begin{proof}
	From $f(1/0)=a/b$ we see that $l_{1/0}<\infty$ if and only if $a$ is odd and $b$ is even.
	When $l_{1/0}<\infty$, applying Proposition \ref{prop:l} and Corollary \ref{cor:distance} to $v=1/0$, $f(v)=a/b$ and $f^2(v)=(a^2+bc)/(b(a+d))$, we get the value of $l_{1/0}$.
	Note that $d(v,f(v))=N(b,a)\equiv b/2\bmod 2$ by Theorem \ref{thm:BW1969}.
	
	For $l_{0/1}$ or $l_{1/1}$, similarly we can take $v=0/1$, $f(v)=c/d$, $f^2(v)=(c(a+d))/(bc+d^2)$
	or $v=1/1$, $f(v)=(a+c)/(b+d)$, $f^2(v)=(a(a+c)+c(b+d))/(b(a+c)+d(b+d))$
	to get the result.
\end{proof}

Denote the $\mathbb{Z}_2$-homology class represented by $T^2\times\{0\}$ in $M_f=T^2\times I/\sim_f$ by $\tau$.
Obviously, we have $||\tau||_{\mathbb{Z}_2}=0$.
The modulo 2 intersection number $\tau \cdot \gamma_x$ is 1 and $\tau \cdot \gamma_\lambda=\tau \cdot \gamma_\mu=0$.
That is to say, $\tau$ is dual to $\gamma_x$.

\begin{prop}
	\label{prop:sum}
	For each $\alpha\in H_2(M_f;\mathbb{Z}_2)$, we have $||\alpha||_{\mathbb{Z}_2}=||\alpha+\tau||_{\mathbb{Z}_2}$.
\end{prop}

Before proving the above proposition, we define the \textbf{sum} of two embedded surfaces.
Suppose that $F_1$ and $F_2$ are two embedded surfaces in a 3-manifold $M$, and $F_1\cap F_2$ consists of simple closed curves $c_1,\cdots,c_n$.
If for each $i=1,\cdots,n$, $c_i$ has a tubular neighborhood $N(c_i)$ homeomorphic to the solid torus, then we can construct an embedded surface, denoted by $F_1+F_2$, such that $\chi(F_1+F_2)=\chi(F_1)+\chi(F_2)$ and $[F_1+F_2]=[F_1]+[F_2]\in H_2(M;\mathbb{Z}_2)$.
In fact, we may assume that each $N(c_i)\cap F_j$ ($i=1,\cdots,n$, $j=1,2$) is either an annulus or a M\"obius band.
Note that $N(c_i)\cap F_1$ and $N(c_i)\cap F_2$ provide a trivialization for the normal bundle of $c_i$, whose total space is homeomorphic to $N(c_i)$.
Since $N(c_i)$ is a solid torus, we must have $N(c_i)\cap F_1\cong N(c_i)\cap F_2$.
In fact, cutting $N(c_i)$ along a meridian disk, up to isotopy we get the upper left or the upper right in Figure \ref{fig:sum}. 
Then we can do surgery in each $N(c_i)$ as in the figure and obtain an embedded surface $F_1+F_2$ in $M$. 
It intersects each $N(c_i)$ at one or two annuli, which implies $\chi(F_1+F_2)=\chi(F_1)+\chi(F_2)$.

\begin{figure}[htbp]
	\centering
	\begin{picture}(350, 190)(0,0)
		\put(0,0){\includegraphics{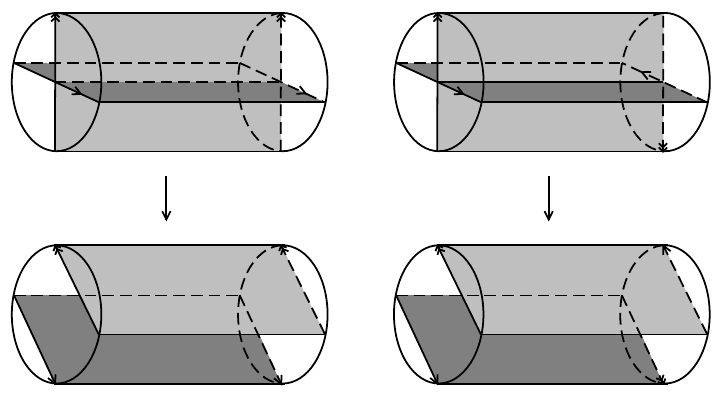}}
		\put(0,108){$N(c_i)$}
		\put(75,154){$c_i$}
		\put(8,145){$F_1$}
		\put(75,122){$F_2$}
		\put(84,93){surgery}
	\end{picture}
	\caption{\label{fig:sum}
		The sum of two embedded surfaces.}
\end{figure}
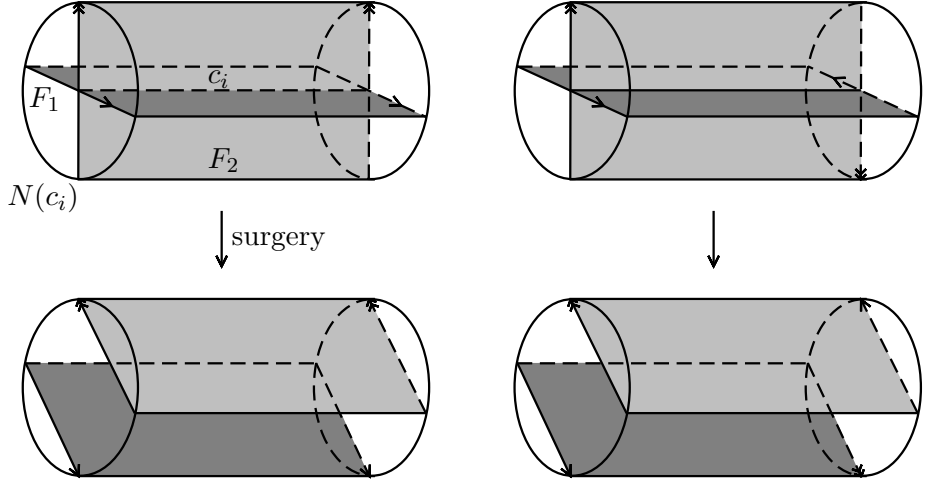

\begin{proof}[Proof of Theorem \ref{prop:sum}]
	We may assume that $\alpha\neq 0,\tau$.
	Let $F$ be a $\mathbb{Z}_2$-taut, incompressible, closed surface in $M_f$ with $[F]=\alpha$.
	By Lemma \ref{lem:fiber}, we may assume that $F$ intersects a torus fiber $T_0$ of $M_f$ at several essential simple closed curves.
	Take $F+T_0$ with $[F+T_0]=\alpha+\tau$ and $\chi(F+T_0)=\chi(F)$.
	Since each component of $F-F\cap T_0$ has non-positive Euler characteristic (by Theorem \ref{thm:thickened-torus}), there is no sphere component in $F+T_0$.
	Therefore, $||\alpha||_{\mathbb{Z}_2}=-\chi(F)=-\chi(F+T_0)\geq ||\alpha+\tau||_{\mathbb{Z}_2}$.
	Also, for $\alpha+\tau$ we have $||\alpha+\tau||_{\mathbb{Z}_2}\geq ||\alpha+\tau+\tau||_{\mathbb{Z}_2}=||\alpha||_{\mathbb{Z}_2}$.
	So we have $||\alpha||_{\mathbb{Z}_2}=||\alpha+\tau||_{\mathbb{Z}_2}$.
\end{proof}

Now the $\mathbb{Z}_2$-Thurston norm on $H_2(M_f;\mathbb{Z}_2)$ can be determined:

\begin{thm}
	\label{thm:bundle}
	For $f\in{\rm Aut}(T^2)$ with matrix 
	$A_f= \left(
	\begin{matrix}
		a & c \\
		b & d
	\end{matrix}
	\right)\in GL(2,\mathbb{Z})$, 
	$H_2(M_f;\mathbb{Z}_2)$ is
	\begin{itemize}
		\item $\mathbb{Z}_2 \langle \tau \rangle$, if 
		$A_f \equiv \left(
		\begin{matrix}
			1 & 1 \\
			1 & 0
		\end{matrix}
		\right)$
		or 
		$\left(
		\begin{matrix}
			0 & 1 \\
			1 & 1
		\end{matrix}
		\right)
		\bmod 2$;
		
		\item $\mathbb{Z}_2 \langle \tau \rangle \oplus \mathbb{Z}_2 \langle [F_{1/0}] \rangle$, if 
		$A_f \equiv \left(
		\begin{matrix}
			1 & 1 \\
			0 & 1
		\end{matrix}
		\right)
		\bmod 2$;
		
		\item $\mathbb{Z}_2 \langle \tau \rangle \oplus \mathbb{Z}_2 \langle [F_{0/1}] \rangle$, if 
		$A_f \equiv \left(
		\begin{matrix}
			1 & 0 \\
			1 & 1
		\end{matrix}
		\right)
		\bmod 2$;
		
		\item $\mathbb{Z}_2 \langle \tau \rangle \oplus \mathbb{Z}_2 \langle [F_{1/1}] \rangle$, if 
		$A_f \equiv \left(
		\begin{matrix}
			0 & 1 \\
			1 & 0
		\end{matrix}
		\right)
		\bmod 2$;
		
		\item $\mathbb{Z}_2 \langle \tau \rangle \oplus \mathbb{Z}_2 \langle [F_{0/1}] \rangle \oplus \mathbb{Z}_2 \langle [F_{1/0}] \rangle$, if 
		$A_f \equiv \left(
		\begin{matrix}
			1 & 0 \\
			0 & 1
		\end{matrix}
		\right)
		\bmod 2$.
	\end{itemize}
	The $\mathbb{Z}_2$-Thurston norms of all elements are: $||0||_{\mathbb{Z}_2}$; $||\tau||_{\mathbb{Z}_2}=0$, realized by a torus fiber; and if $F_{j/k}$ is defined, $||[F_{j/k}]||_{\mathbb{Z}_2}=l_{j/k}$, realized by $F_{j/k}$; $||[F_{j/k}]+\tau||_{\mathbb{Z}_2}=l_{j/k}$, realized by the sum of $F_{j/k}$ and a torus fiber.
\end{thm}

\begin{proof}
	Note that the modulo 2 intersection pairing gives an isomorphism between $H_2(M_f;\mathbb{Z}_2)$ and 
	\begin{displaymath}
		H_1(M_f;\mathbb{Z}_2)\cong \mathbb{Z}_2\langle \gamma_x \rangle \oplus
		\mathbb{Z}_2 \langle \gamma_\lambda \rangle \oplus \mathbb{Z}_2 \langle \gamma_\mu \rangle/\langle \gamma_\lambda-(a\gamma_\lambda+b\gamma_\mu),\gamma_\mu-(c\gamma_\lambda+d\gamma_\mu)\rangle.
	\end{displaymath} 
	
	If 
	$A_f \equiv \left(
	\begin{matrix}
		1 & 1 \\
		1 & 0
	\end{matrix}
	\right)$
	or 
	$\left(
	\begin{matrix}
		0 & 1 \\
		1 & 1
	\end{matrix}
	\right)
	\bmod 2$,
	then $f$ acts on $\{T_{1/0},T_{0/1},T_{1/1}\}$ transitively, and $H_1(M_f;\mathbb{Z}_2) \cong \mathbb{Z}_2 \langle \gamma_x \rangle$.
	Since $\tau$ is dual to $\gamma_x$, we have $H_2(M_f;\mathbb{Z}_2) \cong \mathbb{Z}_2 \langle \tau \rangle$.
	
	If 
	$A_f \equiv \left(
	\begin{matrix}
		1 & 1 \\
		0 & 1
	\end{matrix}
	\right)
	\bmod 2$,
	then $f$ fixes $T_{1/0}$ and exchanges $T_{0/1},T_{1/1}$, and $H_1(M_f;\mathbb{Z}_2) \cong \mathbb{Z}_2 \langle \gamma_x \rangle \oplus \mathbb{Z}_2 \langle \gamma_\mu \rangle$. 
	Since $[F_{1/0}]\cdot \gamma_\mu=1$, according to the modulo 2 intersection pairing we have $H_2(M_f;\mathbb{Z}_2) \cong \mathbb{Z}_2 \langle \tau \rangle \oplus \mathbb{Z}_2 \langle [F_{1/0}] \rangle$.
	
	If 
	$A_f \equiv \left(
	\begin{matrix}
		1 & 0 \\
		1 & 1
	\end{matrix}
	\right)
	\bmod 2$,
	then $f$ fixes $T_{0/1}$ and exchanges $T_{1/0},T_{1/1}$, and $H_1(M_f;\mathbb{Z}_2) \cong \mathbb{Z}_2 \langle \gamma_x \rangle \oplus \mathbb{Z}_2 \langle \gamma_\lambda \rangle$.
	Since $[F_{0/1}]\cdot \gamma_\lambda=1$, we have $H_2(M_f;\mathbb{Z}_2) \cong \mathbb{Z}_2 \langle \tau \rangle \oplus \mathbb{Z}_2 \langle [F_{0/1}] \rangle$.
	
	If 
	$A_f \equiv \left(
	\begin{matrix}
		0 & 1 \\
		1 & 0
	\end{matrix}
	\right)
	\bmod 2$,
	then $f$ fixes $T_{1/1}$ and exchanges $T_{0/1},T_{1/0}$, and $H_1(M_f;\mathbb{Z}_2) \cong \mathbb{Z}_2 \langle \gamma_x \rangle \oplus \mathbb{Z}_2 \langle \gamma_\lambda \rangle$. 
	Since $[F_{1/1}]\cdot \gamma_\lambda=1$, we have $H_2(M_f;\mathbb{Z}_2) \cong \mathbb{Z}_2 \langle \tau \rangle \oplus \mathbb{Z}_2 \langle [F_{1/1}] \rangle$.
	
	If 
	$A_f \equiv \left(
	\begin{matrix}
		1 & 0 \\
		0 & 1
	\end{matrix}
	\right)
	\bmod 2$,
	then $f$ fixes each component of $\mathcal{C}^{i=2}(T^2)$, and
	$H_1(M_f;\mathbb{Z}_2) \cong \mathbb{Z}_2 \langle \gamma_x \rangle \oplus \mathbb{Z}_2 \langle \gamma_\lambda \rangle \oplus \mathbb{Z}_2 \langle \gamma_\mu \rangle$. 
	Therefore, $H_2(M_f;\mathbb{Z}_2) \cong \mathbb{Z}_2 \langle \tau \rangle \oplus \mathbb{Z}_2 \langle [F_{0/1}] \rangle \oplus \mathbb{Z}_2 \langle [F_{1/0}] \rangle$.
	
	By definition, $||0||_{\mathbb{Z}_2}=||\tau||_{\mathbb{Z}_2}=0$.
	And if $F_{j/k}$ is defined, by Proposition \ref{prop:Fjk} we have $||[F_{j/k}]||_{\mathbb{Z}_2}=||[F_{j/k}]+\tau||_{\mathbb{Z}_2}=l_{j/k}$.
\end{proof}

\subsection{Torus semi-bundles}

Now consider torus semi-bundles. 
The universal covering space of any torus semi-bundle is also homeomorphic to $\mathbb{R}^3$.
So torus semi-bundles are irreducible. 
Moreover, since they are orientable and not homeomorphic to $\mathbb{R}P^3$, there is no embedded projective plane in them.

Take $N,K,\lambda,\mu,\alpha$ as in Section \ref{subsect:essential}.
Each torus semi-bundle can be constructed by gluing two copies $N_1,N_2$ of $N$ along their boundaries, hence can be denoted by $N_f=N_1\cup_f N_2$ with $f:\partial N_2\xrightarrow{\cong} \partial N_1$. 
Let $K_i,\lambda_i,\mu_i,\alpha_i\subset N_i$ ($i=1,2$) be the corresponding copies of $K,\lambda,\mu,\alpha\subset N$ respectively.
The isomorphism $f_*:H_1(\partial N_2;\mathbb{Z})\to H_1(\partial N_1;\mathbb{Z})$ corresponds to a matrix in $GL(2,\mathbb{Z})$, also denoted by $A_f$:
\begin{displaymath}
	(\begin{matrix}
		f_*[\lambda_2] & f_*[\mu_2]
	\end{matrix})
	=(\begin{matrix}
		[\lambda_1] &[\mu_1]
	\end{matrix})\,A_f
	=(\begin{matrix}
		[\lambda_1] &[\mu_1]
	\end{matrix})\,
	\left(
	\begin{matrix}
		a & c \\
		b & d
	\end{matrix}
	\right).
\end{displaymath}
The homology groups of $N_f$ are
\begin{displaymath}
	H_1(N_f;\mathbb{Z}) \cong  \dfrac{\mathbb{Z}\langle[\alpha_1]\rangle \oplus \mathbb{Z}_2\langle[\mu_1]\rangle \oplus \mathbb{Z}\langle[\alpha_2]\rangle \oplus \mathbb{Z}_2\langle[\mu_2]\rangle}{\langle 2[\alpha_2]-(2a[\alpha_1]+b[\mu_1]),\, [\mu_2]-(2c[\alpha_2]+d[\mu_1])\rangle},
\end{displaymath}
\begin{displaymath}
	\begin{split}
		H_1(N_f;\mathbb{Z}_2) & \cong  \dfrac{\mathbb{Z}_2\langle[\alpha_1]\rangle \oplus \mathbb{Z}_2\langle[\mu_1]\rangle \oplus \mathbb{Z}_2\langle[\alpha_2]\rangle \oplus \mathbb{Z}_2\langle[\mu_2]\rangle}{\langle -b[\mu_1],\, [\mu_2]-d[\mu_1]\rangle}\\
		& \cong 
		\begin{cases}
			\mathbb{Z}_2\langle[\alpha_1]\rangle \oplus \mathbb{Z}_2\langle[\alpha_2]\rangle, & \text{ if } b \text{ is odd}; \\
			\mathbb{Z}_2\langle[\alpha_1]\rangle \oplus \mathbb{Z}_2\langle[\alpha_2]\rangle \oplus \mathbb{Z}_2\langle[\mu_1]\rangle \text{ with } [\mu_2]=[\mu_1], & \text{ if } b \text{ is even},
		\end{cases}
	\end{split}
\end{displaymath}
and $H_2(N_f;\mathbb{Z}_2)\cong H_1(N_f;\mathbb{Z}_2)$, given by the modulo 2 intersection pairing.
Since $[K_1]\cdot[\alpha_1]=[K_2]\cdot[\alpha_2]=1$, we see that $H_2(N_f;\mathbb{Z}_2)$ can be viewed as a direct sum of two summands: one is $\mathbb{Z}_2\langle[K_1]\rangle \oplus \mathbb{Z}_2\langle[K_2]\rangle$, which is dual to $\mathbb{Z}_2\langle[\alpha_1]\rangle \oplus \mathbb{Z}_2\langle[\alpha_2]\rangle$; and the other is trivial or $\mathbb{Z}_2$ according to whether $b$ is odd or even.

Below we may assume that $b$ is even, and $F$ is a $\mathbb{Z}_2$-taut, incompressible surface in $N_f$ which generates the second summand $\mathbb{Z}_2$ in $H_2(N_f;\mathbb{Z}_2)$ with $[F]\cdot[\mu_1]=[F]\cdot[\mu_2]=1$.
Let $F_2$ be $F\cap N_2$.
It is incompressible in $N_2$.
If it is $\partial$-compressible, we can $\partial$-compress it along a disk $D$, or equivalently, push part of $F$ into $N_1$ along a regular neighborhood of $D$.
By this means, we can make $F_2$ $\partial$-incompressible.
Then we cut $N_1$ along a $\partial$-parallel torus, into a thickened torus $T^2\times I$ and $N_1'\cong N_1$.
With a similar argument, we may assume that the part of $F$ in $N_1'$, denoted by $F_1$, is essential.
Let $F_T$ be the part of $F$ in the thickened torus $T^2\times I$, which is incompressible.
According to Corollary \ref{cor:thickened-torus} and Theorem \ref{thm:twisted-bundle}, the structures of $F_1,F_2,F_T$ are known. 
We now discuss them case by case.
Since $[F]\cdot[\mu_1]=[F]\cdot[\mu_2]=1$, we see that in each $F_i$ ($i=1,2$) there is exactly one component of type (4) in Theorem \ref{thm:twisted-bundle}, for components of other types contribute nothing to the module 2 intersection number with $[\mu_i]$.
Then the other components can only be of type (2).
So $F_i$ consists of a M\"obius band and $m_i$ vertical annuli with $m_i\geq 0$.
Now consider $F_T$ in $T^2\times I$.
Its boundary $\partial F_T$ is the union of $2m_1+1$ parallel curves on $T^2\times\{0\}$ and $2m_2+1$ parallel curves on $T^2\times\{1\}$.
Since $T^2\times\{0\}$ and $T^2\times\{1\}$ are identified with $\partial N_1$ and $\partial N_2$ respectively, we may assume that the $2m_1+1$ curves on $T^2\times\{0\}$ project to $1/0$-curves on $T^2$, and the $2m_2+1$ curves on $T^2\times\{1\}$ project to $a/b$-curves.
Moreover, we may assume that among all components of $F_T$ there is no $\partial$-parallel annulus, for otherwise we can eliminate it by pushing it into $N_1$ or $N_2$.
Then according to Corollary \ref{cor:thickened-torus}, there can only be 2 cases:
\begin{enumerate}
	\item $b=0$, $F_T$ is a union of several vertical annuli, and $F$ is a union of several tori and a Klein bottle, with $||[F]||_{\mathbb{Z}_2}=-\chi(F)=0$;
	\item $b$ is a nonzero even integer, $F_T$ is a non-orientable surface with two boundary components, and $F$ is a closed non-orientable surface of genus $N(b,a)+2$, with $||[F]||_{\mathbb{Z}_2}=N(b,a)$.
\end{enumerate}
Conversely, if $b$ is even, we can construct a $\mathbb{Z}_2$-taut, incompressible, non-orientable surface $F_{b/a}$ with $[F_{b/a}]\cdot[\mu_1]=1$:
\begin{enumerate}
	\item if $b=0$, $F_{b/a}$ can be taken as a Klein bottle, which is a union of two M\"obius bands in $N_1,N_2$ of type (4) in Theorem \ref{thm:twisted-bundle};
	\item if $b$ is a nonzero even integer, $F_{b/a}$ can be taken as a non-orientable surface of genus $N(b,a)+2$, which is a union of two M\"obius bands in $N_1',N_2$ of type (4) in Theorem \ref{thm:twisted-bundle} and a non-orientable surface of genus $N(b,a)$ in $T^2\times I$ bounded by a $1/0$-curve on $T^2\times\{0\}$ and a $b/a$-curve on $T^2\times\{1\}$.
\end{enumerate}

\begin{thm}
	\label{thm:semi-bundle}
	For $f:\partial N_2\xrightarrow{\cong} \partial N_1$ with matrix 
	$A_f= \left(
	\begin{matrix}
		a & c \\
		b & d
	\end{matrix}
	\right)\in GL(2,\mathbb{Z})$, 
	$H_2(N_f;\mathbb{Z}_2)$ is
	\begin{itemize}
		\item $\mathbb{Z}_2\langle[K_1]\rangle \oplus \mathbb{Z}_2\langle[K_2]\rangle$, if $b$ is odd;
		
		\item $\mathbb{Z}_2\langle[K_1]\rangle \oplus \mathbb{Z}_2\langle[K_2]\rangle \oplus \mathbb{Z}_2\langle[F_{b/a}]\rangle$, if $b$ is even. 
	\end{itemize}
	If $b$ is odd or $b=0$, the $\mathbb{Z}_2$-Thurston norms are all zero.
	If $b$ is even and $b\neq 0$, then for any $\alpha\in\mathbb{Z}_2\langle[K_1]\rangle \oplus \mathbb{Z}_2\langle[K_2]\rangle$, $||\alpha||_{\mathbb{Z}_2}=0$ and $||[F_{b/a}]||_{\mathbb{Z}_2}=||[F_{b/a}]+\alpha||_{\mathbb{Z}_2}=N(b,a)$.
	The $\mathbb{Z}_2$-Thurston norms of nontrivial elements are realized by $K_1, K_2, K_1\cup K_2, F_{b/a}$ and their sums.
\end{thm}

\begin{proof}
	We only need to show that $||[F_{b/a}]||_{\mathbb{Z}_2}=||[F_{b/a}]+\alpha||_{\mathbb{Z}_2}$ for $\alpha\in\mathbb{Z}_2\langle[K_1]\rangle \oplus \mathbb{Z}_2\langle[K_2]\rangle$.
	The proof is similar to that of Proposition \ref{prop:sum}.
\end{proof}

\section{Embeddability of non-orientable closed surfaces}\label{sect:embeddability}

In this section we consider the following question: which non-orientable closed surfaces can be embedded in a given 3-manifold $M$?
Once $\Pi_g$ is embedded in $M$, we see that each $\Pi_{g+2k}$ with $k$ any positive integer can be embedded in $M$ by attaching handles to the embedded $\Pi_g$.
Therefore, we only need to find the following two values for $M$:
\begin{itemize}
	\item the \textbf{minimum odd genus}
	\begin{displaymath}
		mog(M)\coloneqq\inf\left\{\{g\geq 1:g\text{ is odd and }\Pi_g\text{ can be embedded in }M\}\cup\{\infty\}\right\};
	\end{displaymath}
	\item the \textbf{minimum even genus} 
	\begin{displaymath}
		meg(M)\coloneqq\inf\left\{\{g\geq 2:g\text{ is even and }\Pi_g\text{ can be embedded in }M\}\cup\{\infty\}\right\}.
	\end{displaymath}
\end{itemize}
If $M$ is non-orientable, then $meg(M)=2$. 
In fact, a regular neighborhood of an orientation-reversing loop is a solid Klein bottle and its boundary is an embedded Klein bottle.
For an orientable closed 3-manifold $M$, it is known that $\min \{mog(M),meg(M)\}<\infty$ if and only if $H_2(M;\mathbb{Z}_2)$ is nontrivial \cite[Proposition 2.2]{BW1969}; $mog(M)<\infty$ if and only if there exists some $\alpha\in H^1(M;\mathbb{Z}_2)$ with $\alpha^3\neq 0\in H^3(M;\mathbb{Z}_2)$.
A similar result is that for a general (not necessarily orientable or compact) 3-manifold $M$, if the Steenrod square $H^1(M;\mathbb{Z}_2)\to H^2(M;\mathbb{Z}_2)$ defined by $\alpha\mapsto\alpha^2$ is trivial, then $mog(M)=\infty$ \cite[Corollary 4.6]{End1992}.
If an embedded surface $F$ realizes $mog(M)$, then $F$ must be incompressible.
According to a famous result of Haken \cite{Haken1961} that an incompressible surface in a triangulated irreducible 3-manifold is isotopic to a normal surface, for irreducible 3-manifolds we can compute $mog(M)$ with an algorithm.
For example, from an input triangulation of $M$, the software Regina (\href{https://regina-normal.github.io/}{https://regina-normal.github.io/}) can solve the coordinate equations of normal surfaces to find all incompressible surfaces.
Also, by the following lemma, for irreducible manifolds there is an algorithm to decide whether $meg(M)$ equals $2$.

\begin{lem}
	\label{lem:incompressible}
	Each embedded Klein bottle in an irreducible orientable 3-manifold is incompressible.
\end{lem}

\begin{proof}
	Suppose $M$ is an irreducible orientable 3-manifold and $K$ is an embedded Klein bottle in $M$.
	If $K$ is compressible, then we can compress it along a disk $D$ to obtain a surface $K'$, which is either a sphere or a union of two projective planes.
	
	Assume that $K'=P_1\cup P_2$ with $P_1\cong P_2\cong\mathbb{R}P^2$.
	Take a tubular neighborhood $N(P_1)$ of $P_1$ in $M-P_2$.
	Since $M$ is orientable, $\partial N(P_1)$ is a sphere, which divides $M$ into two parts.
	However, each part contains an embedded projective plane and thus is not a 3-ball.
	This is in contradiction to the assumption that $M$ is irreducible. 
	
	Assume that $K'$ is a sphere, which bounds an embedded 3-ball $B$ in $M$.
	Then $K$ bounds the union of $B$ and a tubular neighborhood of $D$. 
	However, this implies that $M$ is non-orientable, a contradiction.
\end{proof}

\begin{rem}
	Suppose that $M$ is orientable and closed.
	If a projective plane is embedded in $M$, then its regular neighborhood is homeomorphic to $\mathbb{R}P^3$ with a 3-ball removed.
	So $mog(M)=1$ if and only if there is a $\mathbb{R}P^3$ summand in the prime decomposition of $M$.
	If a Klein bottle is embedded in $M$, then its regular neighborhood is homeomorphic to the twisted $I$-bundle over the Klein bottle, whose boundary is a torus.
	So $meg(M)=2$ if and only if either there is a component in the JSJ decomposition of $M$ that is homeomorphic to the twisted $I$-bundle over the Klein bottle, or there is an $S^2\times S^1$ summand in the prime decomposition of $M$.
\end{rem}

According to \cite{BW1969, Jaco1970, End1992, Rannard1996}, when $M$ is a lens space, or $M= F\times S^1$ with $F$ being any surface, the values of $mog(M)$ and $meg(M)$ are known.
We turn to the case where $M$ is a torus bundle over $S^1$ (not necessarily orientable) or a torus semi-bundle.

According to \cite[\S 2]{BW1969}, if two embedded surfaces $F_1,F_2$ in a closed 3-manifold $M$ satisfy $[F_1]=[F_2]\in H_2(M;\mathbb{Z})$, then $\chi(F_1)\equiv\chi(F_2)\bmod 2$.
Particularly, any embedded non-orientable closed surface of an odd genus is nontrivial in $H_2(M;\mathbb{Z}_2)$.
Therefore, if the $\mathbb{Z}_2$-Thurston norm on $H_2(M;\mathbb{Z}_2)$ is known, then $mog(M)$ can be determined.
As a consequence, for torus bundles over $S^1$ and torus semi-bundles constructed as in Section \ref{sect:norm}, we have the following results.

\begin{thm}
	\label{thm:mog}
	For $f\in{\rm Aut}(T^2)$, $mog(M_f)$ is $\infty$ if there is no odd number among $l_{1/0},l_{0/1},l_{1/1}$; otherwise $mog(M_f)$ equals
	\begin{displaymath}
		2+\min\{l_{j/k}:l_{j/k} \text{ is odd},\, j/k\in\{1/0,0/1,1/1\}\}.
	\end{displaymath}
\end{thm}

\begin{thm}
	\label{thm:meg-mog}
	For $f:\partial N_2\xrightarrow{\cong} \partial N_1$ with matrix 
	$A_f= \left(
	\begin{matrix}
		a & c \\
		b & d
	\end{matrix}
	\right)\in GL(2,\mathbb{Z})$, 
	we have $meg(N_f)=2$, and
	\begin{displaymath}
		mog(N_f)=
		\begin{cases}
			N(b,a)+2, & \text{ if } b\equiv 2\bmod 4;\\
			\infty, & \text{ if } b\not\equiv 2\bmod 4.
		\end{cases}
	\end{displaymath}
\end{thm}

Now consider $meg(M_f)$ with $f\in{\rm Aut}(T^2)$. 
If $f$ is orientation-reversing, then $M_f$ is non-orientable, and thus $meg(M_f)=2$.
If $f$ is orientation-preserving, by an isotopy we may assume that $f$ fixes a point $x\in T^2$.
Then $\{x\}\times I/\sim_f$ is a simple closed curve in $M_f$, which intersects any $T^2$ fiber at exactly one point.
An embedded $\Pi_4$ in $M_f$ can be constructed by attaching a handle guided by $\{x\}\times I/\sim_f$ to a $T^2$ fiber as in Figure \ref{fig:handle}.
Therefore, $meg(M_f)$ equals either $2$ or $4$.
To determine $meg(M_f)$ with $f$ orientation-preserving, the only task is to determine whether the Klein bottle can be embedded in $M_f$.

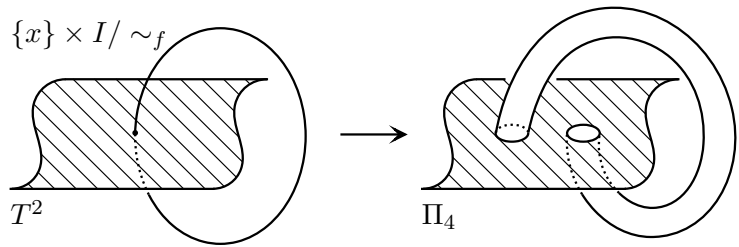
\begin{figure}[htbp]
	\centering
	\begin{picture}(282, 100)(0,0)
		\put(0,0){\includegraphics{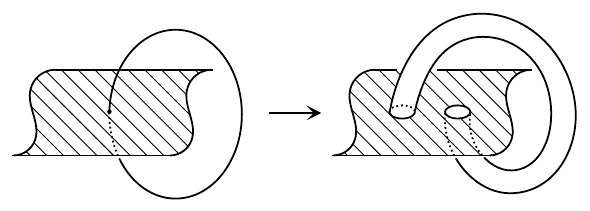}}
		\put(6,13){$T^2$}
		\put(6,82){$\{x\}\times I/\sim_f$}
		\put(160,13){$\Pi_4$}
	\end{picture}
	\caption{\label{fig:handle}
		Attaching a handle along a curve.}
\end{figure}

\begin{thm}
	\label{thm:meg}
	For $f\in{\rm Aut}(T^2)$ with matrix $A_f\in GL(2,\mathbb{Z})$, $meg(M_f)=2$ if $f$ is orientation-reversing or $A_f$ is conjugate to 
	$\left(
	\begin{matrix}
		-1 & 0\\
		n & -1 
	\end{matrix}
	\right)$
	for some $n\in\mathbb{Z}$; otherwise $meg(M_f)=4$.
\end{thm}

\begin{proof}
	If $A_f$ is conjugate to such a matrix, by an isotopy we may assume that $f(\mu)=\mu^{-1}$, where $\mu^{-1}$ represents $\mu$ with the orientation reversed.
	Then the image of $\mu\times I\subset T^2\times I$ in $M_f=T^2\times I/\sim_f$ is an embedded Klein bottle and thus $meg(M_f)=2$.
	
	Below we assume that $f$ is orientation-preserving and $K$ is an embedded Klein bottle in $M_f$.
	By Lemma \ref{lem:incompressible}, $K$ is incompressible.
	And by Lemma \ref{lem:fiber}, there is a torus fiber $T_0$ of $M_f$ such that $K$ intersects $T_0$ at $m$ essential simple closed curves transversely with $m\geq 1$ minimized.
	Without loss of generality, we may assume that $T_0$ is given by the image of $T^2\times\{0\}$.
	Then $K$ is given by an incompressible surfaces in $T^2 \times I$, denoted by $F$.
	There is no $\partial$-parallel annuli in $F$ for otherwise we can push it to decrease the value of $m$.
	Then according to Corollary \ref{cor:thickened-torus}, $F$ can only be a non-orientable surface with two boundary components or a union of several vertical annuli.
	Moreover, by $\chi(F)=\chi(K)=0$, we see that the former case can not happen, and thus without loss of generality, each component of $F$ is parallel to $\mu \times I$.
	If follows that $f$ takes $[\mu]\in H_1(M_f;\mathbb{Z})$ to $\pm [\mu]$. 
	Since $K$ can be obtained from $F$, we must have $f_*([\mu])=-[\mu]$. 
	Then the matrix of $f$ has the form
	$\left(
	\begin{matrix}
		-1 & 0\\
		n & -1 
	\end{matrix}
	\right)$
	with $n\in\mathbb{Z}$.
\end{proof}

The mapping torus $M_f$ of $f\in{\rm Aut}(T^2)$ admits a Seifert fibering if and only if $M_f$ admits either $\mathbb{E}^3$ geometry or Nil geometry, i.e., $A_f$ is either periodic or conjugate to
$\pm\left(
\begin{matrix}
	1 & 0 \\
	n & 1
\end{matrix}
\right)$
for some positive integer $n$ \cite[Theorem 5.5]{Scott1983}.
For them we give the values of $meg(M_f)$ and $mog(M_f)$ directly.

\begin{lem}[{\cite[Lemma 3.4]{Hempel1975}}]
	\label{lem:periodic}
	Up to conjugacy, there are $7$ periodic matrices in $GL(2,\mathbb{Z})$:
	\begin{displaymath}
		A_1=\left(
		\begin{matrix}
			1 & 0 \\
			0 & 1
		\end{matrix}
		\right),
		A_2=\left(
		\begin{matrix}
			-1 & 0 \\
			0 & -1
		\end{matrix}
		\right),
		A_3=\left(
		\begin{matrix}
			1 & 0 \\
			0 & -1
		\end{matrix}
		\right),
		A_4=\left(
		\begin{matrix}
			1 & 0 \\
			1 & -1
		\end{matrix}
		\right),
	\end{displaymath}
	\begin{displaymath}
		A_5=\left(
		\begin{matrix}
			0 & 1 \\
			-1 & -1
		\end{matrix}
		\right),
		A_6=\left(
		\begin{matrix}
			0 & -1 \\
			1 & 0
		\end{matrix}
		\right),
		A_7=\left(
		\begin{matrix}
			0 & 1 \\
			-1 & 1
		\end{matrix}
		\right),
	\end{displaymath}
	which are of orders $1,2,2,2,3,4,6$ respectively.
\end{lem}

\begin{cor}
	Suppose that the matrix $A_f$ of $f\in{\rm Aut}(T^2)$ is periodic.
	
	(1) If $A_f$ is of period $2$, then $meg(m_f)=2$; otherwise $meg(M_f)=4$.
	
	(2) If $A_f$ is conjugate to the matrix $A_3$ or $A_6$ in Lemma \ref{lem:periodic}, then $mog(M_f)=3$; otherwise $mog(M_f)=\infty$.
\end{cor}

\begin{proof}
	Up to conjugacy, we may assume that $A_f$ equals some $A_k$ in Lemma \ref{lem:periodic} with $1\leq k\leq 7$.
	Then (2) follows from Theorem \ref{thm:mog} and Corollary \ref{cor:l}.
	And (1) follows from Theorem \ref{thm:meg}.
	Note that for $k=1,5,6,7$, we have ${\rm tr}A_k\neq -2$ and thus $A_k$ is not conjugate to  
	$\left(
	\begin{matrix}
		-1 & 0\\
		n & -1 
	\end{matrix}
	\right)$
	for any $n\in\mathbb{Z}$.
\end{proof}

Similarly, for any Nil $M_f$, we have the following conclusion.

\begin{prop}
	Suppose that $n$ is a positive integer and $B_n,B_n'$ are the following matrices:
	\begin{displaymath}
		B_n=\left(
		\begin{matrix}
			1 & 0 \\
			n & 1
		\end{matrix}
		\right),\,
		B_n'=
		\left(
		\begin{matrix}
			-1 & 0 \\
			n & -1
		\end{matrix}
		\right).
	\end{displaymath}
	If the matrix $A_f$ of $f\in{\rm Aut}(T^2)$ is conjugate to $B_n$, then $meg(M_f)=4$; if $A_f$ of $f\in{\rm Aut}(T^2)$ is conjugate to $B_n'$, then $meg(M_f)=2$.
	In either case, $mog(M_f)$ equals $\infty$ if $n\equiv 0,1,3\bmod 4$, and equals $n/2+2$ if $n\equiv 2\bmod 4$.
\end{prop}

\bibliographystyle{alpha}

\end{document}